\documentclass[12pt]{amsart}
\usepackage{amsfonts}
\usepackage{}
\usepackage{amsmath,amssymb,amsbsy,amsfonts,amsthm,latexsym,
                        amsopn,amstext,amsxtra,euscript,amscd,mathrsfs,color,bm,cite}
\usepackage{float}
\usepackage[english]{babel}
\usepackage{mathtools}
\usepackage{todonotes}
\usepackage{url}
\usepackage[colorlinks,linkcolor=blue,anchorcolor=blue,citecolor=blue,backref=page]{hyperref}
\usepackage{cases}
\usepackage{txfonts}
\usepackage{verbatim}
\usepackage{subfigure}
\usepackage{diagbox}
\usepackage{array}
\usepackage{bm}

\def\opn#1#2{\def#1{\operatorname{#2}}} 
\opn\chara{char} \opn\length{\ell}
\opn\projdim{proj\,dim} \opn\injdim{inj\,dim} \opn\rank{rank}
\opn\depth{depth} \opn\grade{grade} \opn\height{height}
\opn\embdim{emb\,dim} \opn\codim{codim}

\opn\Tr{Tr} \opn\bigrank{big\,rank}
\opn\superheight{superheight}\opn\lcm{lcm}
\opn\trdeg{tr\,deg}%
\opn\reg{reg} \opn\lreg{lreg}
\opn\div{div} \opn\Div{Div} \opn\cl{cl} \opn\Cl{Cl}
\opn\Spec{Spec} \opn\Supp{Supp} \opn\supp{supp} \opn\Sing{Sing}
\opn\Ass{Ass}
\opn\Ann{Ann} \opn\Rad{Rad} \opn\Soc{Soc}
\opn\Ker{Ker} \opn\Coker{Coker} \opn\Im{Im} \opn\Hom{Hom}
\opn\Tor{Tor} \opn\Ext{Ext} \opn\End{End} \opn\Aut{Aut} \opn\id{id}

\opn\nat{nat}
\opn\pff{pf}
\opn\Pf{Pf} \opn\GL{GL} \opn\SL{SL} \opn\mod{mod} \opn\ord{ord}
\opn\aff{aff} \opn\con{conv} \opn\relint{relint} \opn\st{st}
\opn\lk{lk} \opn\cn{cn} \opn\core{core} \opn\vol{vol}
\opn\gr{gr}

\def\pot#1#2{#1[\kern-0.28ex[#2]\kern-0.28ex]}

\opn\dirlim{\underrightarrow{\lim}}
\opn\invlim{\underleftarrow{\lim}}
\theoremstyle{definition}

\newtheorem*{notation}{Notation}

\theoremstyle{plain}
\newtheorem{theorem}{Theorem}

\newtheorem{lemma}[theorem]{Lemma}

\numberwithin{equation}{section}
\numberwithin{theorem}{section}
\newtheorem{proposition}[theorem]{Proposition}
\newtheorem*{cond}{Condition A}
\def\qed{\ifhmode\textqed\fi
   \ifmmode\ifinner\quad\qedsymbol\else\dispqed\fi\fi}
\def\textqed{\unskip\nobreak\penalty50
    \hskip2em\hbox{}\nobreak\hfil\qedsymbol
    \parfillskip=0pt \finalhyphendemerits=0}
\def\dispqed{\rlap{\qquad\qedsymbol}}

\def\mA{\mathcal{A}}
\def\mV{\mathcal{V}}

\def\mG{\mathcal{G}}
\def\mH{\mathcal{H}}

\def\mD{\mathcal{D}}

\let\ve=\varepsilon

\let\ol=\overline

\begin{document}
\title{B\MakeLowercase{inomial pmf among arithmetic progressions and sieved sets in random walks}}
\author{Jun Hong, Xiaosheng Wu, and Shixin Zhu}
\date{}
\address {School of Mathematics, Hefei University of Technology, Hefei 230009,
P. R. China.}
\email {junhong@mail.hfut.edu.cn}
\address {School of Mathematics, Hefei University of Technology, Hefei 230009,
P. R. China.}
\email {xswu@amss.ac.cn}
\address {School of Mathematics, Hefei University of Technology, Hefei 230009,
P. R. China.}
\email {zhushixin@hfut.edu.cn}


\begin{abstract}
We consider the distribution of the binomial probability mass function (\emph{pmf}) among arithmetic progressions and obtain an average-type theorem. As applications, we consider the possible visits to a kind of sieved sets of integers or lattice points, by an $\alpha$-random walker. We show that, almost surely, the asymptotic proportion of time the random walker in a sieved set of the type is equal to the density of the set, independently of $\alpha$.
\end{abstract}
\maketitle

\section{Introduction}
Random walks are widely concerned recently, and similar phenomena in number theory are also studied. For example, Lifshits and Weber \cite{LW09} and Srichan \cite{Sri15} considered the Lindel\"of hypothesis with the Cauchy random walk; Jouve \cite{Jou10} studied the large sieve method with random walks on cosets of an arithmetic group; McNew \cite{Mcn17} considered random walks on the residues modulo $n$.

In 2019, Cilleruelo, Fern\'andez, and Fern\'andez \cite{CFF19} considered visible lattice points in $\alpha$-random walks. They proved that, almost surely, the asymptotic proportion of time the random walker at visible lattice points is equal to the density of visible points, independently of $\alpha$. Their work relies on the second-moment method in probability, as well as the following asymptotic formula for the sum of the binomial probability mass function (\emph{pmf}) on an arithmetic sequence:
	\begin{align}\label{eq1}
		\left\lvert \sum_{l\equiv v (\bmod a)} \binom nl \alpha^l (1-\alpha)^{n-l}- \frac 1a \right\rvert = O \left(n^{-\frac12}\right).
	\end{align}
Frequently, we need an estimate of the following type,
\begin{align}\label{eq2}
		\sum_{a\in \mathcal{A}}\left\lvert \sum_{l\equiv v (\bmod a)} \binom nl \alpha^l (1-\alpha)^{n-l}- \frac 1a \right\rvert = O \left(\log^{-A}n\right)
	\end{align}
with $\mathcal{A}$ being a set as large as possible, which represents the distribution of the binomial \emph{pmf} among arithmetic progressions.

In \cite{CFF19}, Cilleruelo et al. applied \eqref{eq2} with $\mathcal{A}$ being the set of all factors of $n$, and the asymptotic formula \eqref{eq1} allows for \eqref{eq2} with $|\mathcal{A}|\ll n^{\frac12}\log^{-A}n$.
In this work, we deduce an average-type theorem for \eqref{eq2}, which allows for $|\mA|\ll n\log^{-A}n$. As applications, we can consider a kind of sieved set among positive integers or lattice points in $\alpha$-random walks. We find that the asymptotic proportion of time the random walker in such a sieved set is equal to the density of the set among positive integers or lattice points, independently of $\alpha$. That is to say, the densities of the sieved sets do not alter in $\alpha$-random walks.

\begin{notation}
We highlight here some notations, which we will often use in this work.
For a set $\mA$, we write $|\mA|$ for its cardinality, $\mA(x)=\mA\cap[1,x]$, and $1_{\mA}(n)$ for its indicative function, which equals to $1$ if $n\in\mA$ and vanishes otherwise. In particular, we write $1_x(n)=1_{[1,x]}(n)$ for convenience. We also present some well-known arithmetic functions: $\tau_k(n)$ means the divisor function; $\lfloor x \rfloor$ denotes the largest integer not exceeding $x$. We write $(m,n)$ for the greatest common divisor of $m$ and $n$, but sometimes, it also means a lattice point, and it is easy to identify them in the work. As a convention, we apply $A$ for a sufficiently large positive constant and $\ve$ for a sufficiently small positive constant, which may vary from line to line.
\end{notation}

\subsection{The distribution of the binomial \emph{pmf} among arithmetic progressions}
We consider the average distribution of the binomial \emph{pmf} among arithmetic progressions, while the moduli take values in a set, which is not too dense among positive integers and meets the following condition.
\begin{cond}Let $\theta\in (0,1]$ be a given constant. For $x>0$, we write $\mA_x$ for a set of positive integers meeting:
\begin{itemize}
\item $a\ll x^A$, for any $a\in \mA_x$;
\item $|\mA_x\cap[1,y]|\ll \frac{y^{\theta}}{\log^A y}$, for any $y\gg x^\ve$.
\end{itemize}
\end{cond}
We have the following theorem for the average distribution of the binomial \emph{pmf} among arithmetic progressions.
\begin{theorem}\label{thmone}
Suppose by $\mA_x$ a set meeting Condition A with $\theta\in(0,1]$. For any $\alpha\in (0,1)$, it holds uniformly in $x'\in (x^\ve,x)$ and $h\in (-Ax,\alpha' x'/2)$ that
	\begin{align}\label{eqaB}
		\frac1x\sum_{x'\le n\le x}\sum_{a\in \mA_x}\Bigg| \sum_{l\equiv v (\bmod a)} \binom nl \alpha^l (1-\alpha)^{n-l}- \frac 1{a} \Bigg|=O\left(x^{-\frac12(1-\theta)}(\log x)^{-A}\right),
	\end{align}
where $\alpha'=\alpha$ for $v=h$ and $\alpha'=1-\alpha$ for  $v=n-h$.
\end{theorem}
For a vector $\bm{a}=(a_1,\dots,a_u)^T$, we write $\Pi\bm{a}$ for the product of its coordinates, that is to say,
\[
\Pi\bm{a}=\prod_{t=1}^u a_t.
\]
If $\bm{v}=(v_1,\dots,v_u)^T$ is another vector, the notation $l\equiv\bm{v} \pmod{\bm{a}}$ means $l\equiv v_t \pmod{a_t}$ for all $t=1,\dots,u$.
We also have the following theorem with multiple moduli.
\begin{theorem}\label{thmaverageB}
Suppose by $\mA^1_x,\dots, \mA^u_x$ the sets meeting Condition A with $\theta\in(0,1)$, and denote by $\mD_n$ a set consisting of $\ll n^\ve$ elements, which is uniquely determined by $n$. For any $\alpha\in (0,1)$ and any given $\nu$,
it holds uniformly in $x'\in (x^\ve,x)$ and $h_i\in (-Ax,\alpha' x'/2)$ that
	\begin{align}\label{eqaB1}
		\frac1x\sum_{x'\le n\le x}\sum_{d\in \mD_n}\mathop{\sum\nolimits'}_{a_1\in \mA^1_x,\dots,a_u\in \mA^u_x}\Bigg| \sum_{\substack{l\equiv \bm{v} (\bmod \bm{a})\\ l\equiv \nu (\bmod d)}} \binom nl \alpha^l (1-\alpha)^{n-l}- \frac 1{ d\Pi\bm{a}} \Bigg|=O\left(x^{-\frac1{2u}(1-\theta)+\ve}\right),
	\end{align}
where $\alpha'=\alpha$ for $v_i=h_i$ and $\alpha'=1-\alpha$ for  $v_i=n-h_i$, and where $\sum'$ is the sum over such $a_1,\dots, a_u$ that are coprime with each other and $d$.
\end{theorem}
It is worth remarking that Theorem \ref{thmaverageB} is not a direct result of Theorem \ref{thmone} by combining $d, a_1,\dots,a_u$ to a variable $a$ with the Chinese Remainder Theorem. The corresponding residue $v$ in applying the Chinese Remainder Theorem would depend on $a$ and $n$, while Theorem \ref{thmone} does not apply to this case.
Theorems \ref{thmone} and \ref{thmaverageB} may have applications on many occasions, and we consider their applications to studying sieved sets in $\alpha$-random walks here.

Let $\mV$ be a subset of $\mathbb{N}$ (or $\mathbb{N}^2$).
We write $P_i$ for the $i$-step of the random walker in $\mathbb{N}$ (or $\mathbb{N}^2$) and apply a sequence $(X_i)_{i\ge 1}$ given by
\begin{equation}\label{eqdefXi}
	X_i=\begin{cases}
		1, & if\ P_i\in \mV,\\
		0, & if\ \text{not}.
	\end{cases}
\end{equation}
Then the proportion of time that $\mV$ is visited by the random walker in the first $n$ steps is
\begin{align}\label{eqdefSn}
	\overline{S}_n(\mV)=\frac {X_1+X_2+\cdots +X_n}n.
\end{align}
Denote by
\[
\lim_{n\rightarrow\infty}\overline{S}_n(\mV)
\]
the density of the sieved set $\mV$ in random walks. We can consider the asymptotic proportion of time, that $\mV$ is visited by a random walker, by calculating the density.

\subsection{Sieved sets of integers and random walks}\label{secint}
Let $\mH=\{h_1,\dots,h_u\}$ be a set of $u$ integers, and let $\mG$ be a subset of $\mathbb{N}$, whose elements are coprime with each other.
We consider a set sieved by $\mG$ with respect to $\mH$, which is defined by
\begin{align}\label{eqv1}
\mV:=\{n\in \mathbb{N}: n\nequiv h_t ~(\bmod ~g), \ \forall  t\le u, \ \forall g\in\mG\}.
\end{align}
Denote by $\nu_g(\mH)$ the number of distinct residue classes modulo $g$ occupied by all elements of $\mH$. It is obvious that $\mV$ is empty if $\nu_g(\mH)=g$ for some $g\in\mG$.
To avoid this, we require
\[
\nu_g(\mH)<g,  \ \forall g\in \mG,
\]
and we say that $\mH$ is \emph{admissible} with respect to $\mG$ if the condition is meeted.

There are many well-known examples of this type, such as the classical set sieved by small primes and the set of square free numbers. We consider this type of sieved set in $\alpha$-random walks.
For $\alpha\in (0,1)$, an $\alpha$-random walk in $\mathbb{N}$, starting at $P_0=0$, is given by
\begin{align}\label{eq1rw}
P_i=P_{i-1}+W_i,\ \ \ \ \ i=1,2,\dots,
\end{align}
with $W_i$ being an independent Bernoulli variable with success probability $\alpha$ that
\begin{align}
W_i=\begin{cases}
1,&\ \text{with probability} ~\alpha,\\
0,&\ \text{with probability}~ 1-\alpha.
\end{cases}
\end{align}
If $\mG$ is not ``too dense", we find that the density of the sieved set in $\alpha$-random walks is equal to its density among $\mathbb{N}$, independently of $\alpha$.

Our first result is about the classical set sieved by small primes, and this is the only case, where we need to apply Theorem \ref{thmone} with $\theta=1$.
For a given constant $0<\beta<1/2$,
we consider the following set
\[
\mV^c(x)=\left\{n\le x: (n,\prod_{p\le \exp(\log^\beta x)}p)=1\right\}.
\]
It is well-known that
\begin{align}\label{eqUx}
|\mV^c(x)|= c_x (1+o(1))x
\end{align}
with
\begin{align}\label{eqcx}
c_x=\prod_{p\le\exp(\log^\beta x)}(1-p^{-1}).
\end{align}
The density of $\mV^c(x)$ among positive integers tends to zero since $c_x\rightarrow0$ as $x\rightarrow\infty$.
By Theorem \ref{thmone},
we find that the asymptotic proportion of time that $\mV^c(n)$ is visited in the first $n$ steps by a random walker has the same asymptotic formula.
\begin{theorem}\label{thmsmallp}
Let $P_i$ be an $\alpha$-random walk as in \eqref{eq1rw}. For any $\alpha \in (0,1)$, we have that
	\begin{align*}
		\overline{S}_n(\mV^c(n))=c_n (1+o(1))
	\end{align*}
	almost surely, where $c_n$ is given by \eqref{eqcx}.
\end{theorem}

Denote by
\begin{align}\label{eqdefA}
\mA=\{a\in \mathbb{N}: a=1\ \text{or}\ a=g_1g_2\cdots g_s\ \text{with}\ g_i\in\mG,\ g_i\neq g_j,\ s\ge1\}
\end{align}
a set consisting of products of distinct elements in $\mG$.
We write
\[
\mA_x=\{a\in\mA: a\le x^A\}.
\]
If $\mG$ is not ``too dense", the set $\mA$
could meet Condition A with $\theta<1$ for sufficiently large $x$. A well-known sieved set of this type is the set of $k$-free numbers.
With $\mG=\{p^2: p \ \text{is prime}\}$ and $\mH=\{0\}$, $\mV$ denotes the set of square-free numbers, and it is well-known that
\[
\lim_{x\rightarrow\infty}\frac{|\mV(x)|}{x}=\frac1{\zeta(2)}=\frac{6}{\pi^2}\approx 0.60793,
\]
where $\zeta(s)$ is the Riemann zeta-function.
With $\mH=\{0,1\}$, it evolves into the set of consecutive square-free numbers, and it is proved by Carlitz \cite{Car32} that
\[
\lim_{x\rightarrow\infty}\frac{|\mV(x)|}{x}=\prod_{p}\left(1-\frac2 {p^{2}}\right)+O\left(x^{-\frac13+\ve}\right).
\]
With $\mH=\{0,1,2\}$, it further involves into the set of consecutive square-free numbers of length $3$. For a general sieved set of this type, we can follow the way of Carlitz \cite{Car32} to deduce its density in the following theorem.
\begin{theorem}\label{pro1v}
Let $\mH=\{h_1,\dots,h_u\}$ be an admissible set with respect to $\mG$, and $\mV$ is a sieved set defined by \eqref{eqv1}. If $\mA_x$ meets Condition A with $\theta<1$, we have
\begin{align}
|\mV(x)|=cx+O\left(x^{\frac{u\theta}{u\theta-\theta+1}+\ve}\right)
\end{align}
with
\begin{align}\label{eqc}
c=\prod_{g\in\mG}\left(1-\frac{\nu_g(\mH)}{g}\right).
\end{align}
\end{theorem}

With $\theta=1/2$ and $u=2$ in Theorem \ref{pro1v}, it gives the result of Carlitz \cite{Car32}. Heath-Brown \cite{HB84} has sharpened the error term of Carlitz's result to $O\left(x^{-\frac{4}{11}}\log^7x\right)$ by introducing the celebrated square sieve method.
However, we do not pursue this direction here since we only care about its density in $\alpha$-random walks.


For a sieved set of this type, we can deduce its density in $\alpha$-random walks with the help of Theorem \ref{thmaverageB}.

\begin{theorem}\label{thm2}
With $\mH$ and $\mV$ being as in Theorem \ref{pro1v}, let $P_i$ be an $\alpha$-random walk as in \eqref{eq1rw}. If $\mA_n$ meets Condition A with $\theta<1$, we have, for any $\alpha \in (0,1)$, that
	\begin{align*}
		\lim_{n\to\infty} \overline{S}_n(\mV)=c
	\end{align*}
	almost surely, where $c$ is the constant given by \eqref{eqc}.
\end{theorem}

\subsection{Application to sieved sets of lattice points}\label{secv2}
Let $\mG_1$ and $\mG_2$ be two sets as $\mG$, and $\mH_1=\{h^1_1,\dots,h^1_u\}$ and $\mH_2=\{h^2_1,\dots,h^2_w\}$ are admissible with respect to $\mG_1$ and $\mG_2$ respectively. As in \eqref{eqdefA}, we have $\mA^1_x$ and $\mA^2_x$ according to $\mG_1$ and $\mG_2$.
We consider the following sieved set of lattice points defined via
\begin{align}\label{eqv2}
\mV^2:=\{(m,n)\in \mathbb{N}^2: m\in\mV_1, n\in \mV_2\}
\end{align}
with $\mV_i$ being defined as in \eqref{eqv1}, which are subsets of $\mathbb{N}$ sieved by $\mG_i$ with respect to $\mH_i$. Also, we write
 \[
 \mV^2(x)=\mV^2\cap[1,x]\times[1,x].
 \]
 As an example, $\mV^2$ turns out to be the set of square-free lattice points if $\mG_1=\mG_2=\{p^2: p\ \text{is prime}\}$ with $\mH_1=\mH_2=\{0\}$, and its density is known as $\zeta(2)^{-2}$.

For the density of $\mV^2$ among lattice points, a direct result of Theorem \ref{pro1v} shows that
\[
\left|\mV^2(x)\right|=c_1c_2x^2+O\left(\left(x^{\frac{u\theta}{u\theta-\theta+1}}+x^{\frac{w\theta}{w\theta-\theta+1}}\right)x^{1+\ve}\right),
\]
with
\begin{align}\label{eqc1c2}
c_1=\prod_{g\in\mG_1}\left(1-\frac{\nu_g(\mH_1)}{g}\right),\ \ \ \ \ c_2=\prod_{g\in\mG_2}\left(1-\frac{\nu_g(\mH_2)}{g}\right).
\end{align}

For $0<\alpha<1$, an $\alpha$-random walk in the two-dimensional lattice, starting from $P_0=(0,0)$, is defined by
\begin{equation}\label{eq2rw}
	P_{i+1}=P_i+\begin{cases}
		(1,0), & with\  probability\  \alpha,\\
		(0,1), & with\  probability\  1-\alpha,
	\end{cases}
\end{equation}
where $P_i=(x_i,y_i)$ is the coordinate of the $i$-th step of the $\alpha$-random walker for $i=0,1,2,\cdots$.
We find that the density of $\mV^2$ in random walks is equal to the density of $\mV^2$ among lattice points, independently of $\alpha$.

\begin{theorem}\label{thm3}
Let $\mV^2$ be a sieved set given by \eqref{eqv2}, and let $P_i$ be an $\alpha$-random walk as in \eqref{eq2rw}.
If $\mA^1_n$ and $\mA^2_n$ meet Condition A with $\theta_1, \theta_2<1$, we have, for any $\alpha\in(0,1)$, that
\[
\lim_{n\to\infty} \overline{S}_n(\mV^2)=c_1c_2
\]
almost surely, where $c_1$ and $c_2$ are constants given by \eqref{eqc1c2}.
\end{theorem}

\subsection{Sieved set among visible lattice points}\label{secvl}
A lattice point $(m,n)$ is visible from the origin, or simply visible, if $(m,n)=1$. It is well-known that the asymptotic proportion of visible lattice points is $\zeta(2)^{-1}$; See also \cite[Section 3.8]{Ap76}.
The structure of the distribution of visible lattice points and its generalizations have been studied in many aspects, such as, its angular distribution, diffraction pattern, ergodic properties and so on. There are a large number of investigations, and one may see Boca et al. \cite{BCZ00}, Baake et al. \cite{BMP00}, Baake and Huck\cite{BH15}, Adhikari and Granville \cite{AG09}, Chen and Cheng \cite{CC03}, Liu and Meng \cite{LM20} as examples.


Let $\mV_1$, $\mV_2$ be two sieved sets of integers as before.
We further consider the following sieved set of visible lattice points defined via
\begin{align}\label{eqv3}
\mV^3:=\{(m,n)\in \mathbb{N}^2: m\in\mV_1, n\in \mV_2 ~\text{and} ~(m,n)=1\}
\end{align}
and
\[
\mV^3(x)=\mV^3\cap[1,x]\times[1,x].
\]
Before this, some notations would be presented here for ease of reference.
For an integer $r$, denote by $\kappa_i(r)$ the number of $h^i_t\in\mH_i$, satisfying $h^i_t\equiv0 \pmod r$.
We write $\rho_i(r)$ for the product of all $g\in\mG_i$ with $(g,r)>1$, that is to say,
\begin{align}\label{eqgi}
\rho_i(r)=\prod\limits_{\substack{g\in\mG_i\\(g,r)>1}}g.
\end{align}
We further apply $\omega_i(r)$ for the product
\begin{align}\label{eqomega}
\omega_i(r)=\prod\limits_{g\mid \rho_i(r)}\frac{-\kappa_i((g,r))}{g}.
\end{align}
For a given set $\mH_i=\{h^i_1,\dots, h^i_{u}\}$, we write $\bm{h_i}$ for the vector $\bm{h_i}=(h^i_1,\dots, h^i_{u})^T$.

\begin{theorem}\label{thm4}
If $\mA^1_x$ and $\mA^2_x$ meet Condition A with $\theta_1, \theta_2<1$, we have
\begin{align*}
\left|\mV^3(x)\right|=c_1c_2c_3x^2 +O\left(\left(x^{\frac{u\theta}{u\theta-\theta+1}}+x^{\frac{w\theta}{w\theta-\theta+1}}\right)x^{1+\ve}\right)
\end{align*}
with $c_1$, $c_2$ as in \eqref{eqc1c2} and $c_3$ being another constant defined via
\begin{align}\label{eqc3}
c_3=\sum_r\frac{\mu(r)}{r^2}f_1(r)f_2(r)
\end{align}
with
\begin{align}\label{eqfi}
f_i(r)=\prod_{g\mid \rho_i(r)}\left(1-\frac{\nu_g(\mH_i)}{g}\right)^{-1}\sum_{d\mid (r,\Pi\bm{h_i})}d\omega_i(d)
\end{align}
for $i=1,2$.
\end{theorem}

In general, there does not exist an Euler product for
\[
\sum_r\frac{\mu(r)}{r^2}f_1(r)f_2(r),
\]
but one may deduce one if both $\mG_1$ and $\mG_2$ are consisting of primes or prime powers. For example, let us consider the case of $k$-th power free numbers, where
\[
\mG_1=\{p^{k_1}, p ~\text{is prime}\},\ \ \ \ \ \mG_2=\{p^{k_2}, p ~\text{is prime}\}
\]
with $k_1, k_2\ge 2$.
It is easy to see from \eqref{eqfi} that
\[
f_1(r)=\prod_{p\mid r}\Bigg(1-\frac{\nu_{p^{k_1}}(\mH_1)}{p^{k_1}}\Bigg)^{-1}\left(1-\frac{\kappa_1(p)}{p^{k_1-1}}\right),
\]
\[
f_2(r)=\prod_{p\mid r}\Bigg(1-\frac{\nu_{p^{k_2}}(\mH_2)}{p^{k_2}}\Bigg)^{-1}\left(1-\frac{\kappa_2(p)}{p^{k_2-1}}\right).
\]
With this, an easy calculation shows
\begin{align}\label{eqc3k}
c_3=\prod_p\left(1-\left(\frac{p^{k_1-1}-\kappa_1(p)}{p^{k_1}-\nu_{p^{k_1}}(\mH_1)}\right)\left( \frac{p^{k_2-1}-\kappa_2(p)}{p^{k_2}-\nu_{p^{k_2}}(\mH_2)}\right)\right).
\end{align}

It should be noted that $c_3=0$ is possible. For example, with $k_1=k_2=2$ and $\mH_1=\mH_2=\{-1,0,1\}$, the factor due to $p=2$ will vanish in \eqref{eqc3k}. On the other hand, an elementary discussion shows that $\mV_3$ is empty in this case. If $n-1,n,n+1$ and $m-1,m,m+1$ are all square-free numbers, all these six numbers should not be divisible by $4$. Thus, there is $m\equiv n\equiv2\pmod4$, and this conflicts with the condition $(m,n)=1$. However, if adjusts the sets to $\mH_1=\mH_2=\{0,1,2\}$, the constant $c_3$ would not vanish.

Cilleruelo et al. \cite{CFF19} considered the possible visits to visible lattice points by an $\alpha$-random walker.
For a deeper look at visible lattice points in $\alpha$-random walks, we consider the possible visits to $\mV^3$, a sieved subset of visible lattice points, by an $\alpha$-random walker. We obtain that, almost surely, the asymptotic proportion of time the random walker in $\mV^3$ is equal to the density of $\mV^3$ too, and thus, the density of $\mV^3$ does not alter in $\alpha$-random walks.

\begin{theorem}\label{thm5}
Let $\mV^3$ be a set of visible lattice points given by \eqref{eqv3}, and let $P_i$ be an $\alpha$-random walk as in \eqref{eq2rw}.
If $\mA^1_n$ and $\mA^2_n$ meet Condition A with $\theta_1, \theta_2<1$ and $c_3\neq 0$, we have, for any $\alpha\in(0,1)$, that
\[
\lim_{n\to\infty} \overline{S}_n(\mV^3)=c_1c_2c_3
\]
almost surely.
\end{theorem}

\section{Preliminaries}
\subsection{Arithmetic results}
Let $\mu(n)$ be the M\"obius function, and the following orthogonality of the M\"obius function is well-known (see also \cite[Section 2.2]{Ap76}).
\begin{lemma}\label{lemmobius}
For any integer $n$,
\begin{equation*}
	\sum_{d\mid n} \mu (d)=\begin{cases}
		1, & \text{if}\ n=1,\\
		0, & \text{otherwise}.
	\end{cases}
\end{equation*}
\end{lemma}
Our next two lemmas are about the set meeting Condition A.
\begin{lemma}\label{lemsumtau}
Let $\mA$ be a set of integers and $\theta\in(0,1]$. If
\[
|\mA(y)|\ll y^{\theta}\log^{-A}y
\]
holds for any $y\ge x$, we have
\begin{align}\label{eqsumtau}
\sum_{y\le a\in \mA }\frac{1}{a}\ll y^{-(1-\theta)} \log^{-A} y
\end{align}
holds uniformly in $y$ with $y\ge x$.
\end{lemma}
\begin{proof}
In virtue of the dyadic partition method, we can split the segment $[y,\infty)$ into small intervals of forms $[M,2M]$ with $ M\ge y$. First, we deduce an upper bound for the sum over $[M,2M]$ by
\begin{align*}
\sum_{a\in \mA\cap[M,2M] }\frac{1}{a}\ll M^{-1}|\mA\cap[M,2M]|\ll M^{-(1-\theta)} \log^{-A} M.
\end{align*}
By summing over all small intervals, we have
\begin{align*}
\sum_{M\le a\in \mA}\frac{1}{a}\ll M^{-(1-\theta)}\log^{-A}M\sum_{k=0}^\infty 2^{-k(1-\theta)}\left(1+\frac{k\log 2}{\log M}\right)^{-A}.
\end{align*}
The convergence of the $k$-sum is obvious when $\theta<1$. For $\theta=1$,
\[
\sum_{k=0}^\infty \left(1+\frac{k\log 2}{\log M}\right)^{-A}\ll \log M.
\]
This establishes the lemma.
\end{proof}

\begin{lemma}\label{lemgamma}
The sets $\mA^i$, for $i=1,2$, are defined as in \eqref{eqdefA}, satisfying
\[
|\mA^i(x)|\ll x^{\theta_i}\log^{-A}x
\]
with $\theta_i\in (0,1)$. Let
\begin{align}\label{eqgamma}
\gammaup_i(d)=\begin{cases}
		\prod\limits_{g\mid \rho_i(d)}\frac{\nu_g(\mH_i)}{g}, & \text{if}\ d\mid a\ \text{for some} \ a\in\mA^i,\\
		0, & \text{otherwise}.
\end{cases}
\end{align}
We have
\[
\sum_{d} d\gammaup_1(d)\gammaup_2(d)\ll 1.
\]
\end{lemma}
\begin{proof}
By the definition, there is
\[
\gammaup_i(d)\ll \rho_i(d)^{-1+\ve}\ll d^{-1+\ve}.
\]
Thus, we have
\[
\sum_{d} d\gammaup_1(d)\gammaup_2(d)\ll \sum_{d}\rho_1(d)^{-1+\ve}\ll \sum_{a\in\mA^1}a^{-1+\ve}\tau(a)\ll 1
\]
by Lemma \ref{lemsumtau}.
\end{proof}
\subsection{Estimates for binomial distribution}
We first present some properties for the binomial \emph{pmf}, which turns out to be a unimodal function decaying rapidly beyond the maximum.
\begin{lemma}\label{lemBio}
Let $\alpha\in (0,1)$, $n\in\mathbb{N}$, and $l_0=\lfloor\alpha(n+1)\rfloor$. 
\begin{enumerate}
  \item  $\binom nl \alpha^l (1-\alpha)^{n-l}\ll n^{-\frac12}$ holds uniformly for $0\le l\le n$;
  \item  $\binom nl \alpha^l (1-\alpha)^{n-l}$ is a unimodal function of $l$, and its maximum is taken at $l_0$;
  \item  For $|l-l_0|\ge\sqrt n \log n$, there is
  \[
  \binom nl \alpha^l (1-\alpha)^{n-l}\ll 2^{-\log^2 n}.
  \]
\end{enumerate}
\end{lemma}
\begin{proof}
The first property is a direct result of the local central limit theorem (see, for instance \cite{DR10}), and the second property is well-known. So we only provide proof for property (3), where we lay focus on the case $l<l_0$, and the other case is identical. Let
\[
h=l_0-l\ge \sqrt n \log n.
\]
Consider the ratio of values taken at $l$ and $l_0$, that is
\begin{align}\label{eqratio}
\frac{\binom nl \alpha^l (1-\alpha)^{n-l}}{\binom n{l_0} \alpha^{l_0} (1-\alpha)^{n-l_0}}=\frac{l_0(l_0-1)\cdots (l_0-h+1)}{(n-l_0+1)(n-l_0+2)\cdots (n-l_0+h)}\alpha^{-h}(1-\alpha)^{h}.
\end{align}
Observe that
\[
\alpha n-1\le l_0\le \alpha n+1,
\]
and we have
\[
\frac{l_0(l_0-1)\cdots (l_0-h+1)}{(\alpha n)^h}\ll \left(1-\frac{\sqrt n\log n}{2\alpha n}\right)^{\frac{\sqrt n\log n}2} \ll 2^{-\frac{\log^2 n}{4\alpha }},
\]
\[
\frac{((1-\alpha) n)^h}{(n-l_0+1)(n-l_0+2)\cdots (n-l_0+h)}\ll \left(1+\frac{\sqrt n\log n}{2(1-\alpha) n}\right)^{-\frac{\sqrt n\log n}2}\ll 2^{-\frac{\log^2 n}{4(1-\alpha)}}.
\]
Applying this into \eqref{eqratio} gives
\[
\frac{\binom nl \alpha^l (1-\alpha)^{n-l}}{\binom n{l_0} \alpha^{l_0} (1-\alpha)^{n-l_0}}\ll 2^{-\log^2 n}
\]
for $h\ge\sqrt{n} \log n$.
This, in combination with the uniform upper bound in property (1), establishes property (3).
\end{proof}

We also need an asymptotic formula for the sum of the binomial \emph{pmf} on an arithmetic sequence.
\begin{lemma}\label{lemaverage}
	For any $0<\alpha<1$ and $n\in\mathbb{N}$, then for any integer $a \ge 1$ and any $v\in\{0,1,\cdots,a-1\}$, we have
	\begin{align*}
		\left\lvert \sum_{l\equiv v (\bmod a)} \binom nl \alpha^l (1-\alpha)^{n-l}- \frac 1a \right\rvert = O \left(n^{-\frac12}\right).
	\end{align*}
\end{lemma}
\begin{proof}
For $a> n$, it is obvious from the uniform upper bound of \emph{pmf}, and for $a\le n$, one may refer to \cite[Lemma 2.1]{CFF19} for a proof.
\end{proof}

\subsection{Second-moment method in probability}
The study of the density of a sieved set in random walks, such as Theorems \ref{thmsmallp}, \ref{thm2}, \ref{thm3}, and \ref{thm5}, is based on a special second-moment method in probability, by which we can transform the problem into calculating the expectation and the variance.
We present the special second-moment method in the following lemma, which is a generalization of \cite[Lemma 2.6]{CFF19}.
\begin{lemma}\label{lemsmm}
	Let $W_i,\ i=1,2,\cdots$ be a sequence of uniformly bounded random variables, and let $\overline{S}_n=(W_1+W_2+\cdots+W_n)/n$. If $\mathbb{E}(\overline{S}_n)\gg \log^{-B} n$, $\mathbb{V}(\overline{S}_n)=O\left(\log^{-C}n\right)$ for some positive constants $B, C$, satisfying $C>3B+3$, and
\begin{align}\label{eqsm1}
\lim_{n\rightarrow\infty}\frac{\mathbb{E}(\overline{S}_n)-\mathbb{E}(\overline{S}_{n-n'})}{\mathbb{E}(\overline{S}_n)}=0 \ \ \text{holds uniformly for}\ \ n'\ll n \log^{-B} n,
\end{align}
then we have
	\begin{align*}
		\lim_{n\rightarrow\infty}\overline{S}_n /\mathbb{E} (\overline{S}_n)=1
	\end{align*}
	almost surely. In particular, if
\[
\lim_{n\rightarrow\infty}\mathbb{E}(\overline{S}_n)=\delta\neq0,
\]
we have
\[
\lim_{n\rightarrow\infty}\overline{S}_n=\delta
\]
almost surely.
\end{lemma}
\begin{proof}
Let $\sigma$ be a constant satisfying $3/C<\sigma<1/(B+1)$ and $k(m)=\exp\left(m^\sigma\right)$. For any large $n$, there is a unique $m$ such that $k(m)\le n<k(m+1)$, and it is easy to see that
\begin{align}\label{+46}
n-k(m)\le \exp(m^\sigma)\left(\exp((m+1)^{\sigma}-m^\sigma)-1\right)\ll nm^{\sigma-1}\ll n\log^{-B}n.
\end{align}
Since \begin{align*}
\mathbb{V}(\overline{S}_{k(m)})\ll \log^{-C}k(m)\ll m^{-\sigma C},
\end{align*}
we have, by the Chebyshev inequality, that
\[
\sum_{m=1}^\infty\mathbb{P}\left(\left|\ol{S}_{k(m)}-\mathbb{E}(\ol{S}_{k(m)})\ge\frac1{m}\right|\right)\le\sum_{m=1}^\infty m^2\mathbb{V}(\ol{S}_{k(m)})\ll\sum_{m=1}^\infty m^{-\sigma C+2}<\infty.
\]
By this, we have from the Borel-Cantelli lemma (see \cite[Section 2.3]{DR10}) that
\begin{align}\label{eqsm2}
\mathop{\lim\sup}_{m\rightarrow\infty}m\left|\ol{S}_{k(m)}-\mathbb{E}(\ol{S}_{k(m)})\right|\le 1
\end{align}
almost surely. To extend this for any $n$, we also need the following simple fact: if $(x_i)_{i\ge1}$ is a numerical sequence satisfying $|x_i|\le 1$, then for any $s\le t$,
\[
\left|\frac1t\sum_{i\le t}x_i-\frac1s\sum_{i\le s}x_i\right|\le 2\left(1-\frac st\right).
\]
Hence, we have that
\[
\ol{S}_{n}-\ol{S}_{k(m)}\ll 1-\frac{k(m)}{n}\ll m^{\sigma-1}
\]
by \eqref{+46}. This, in combination with \eqref{eqsm2}, gives
\begin{align*}
\overline{S}_n -\mathbb{E} (\overline{S}_n)&\le\left|\ol{S}_{n}-\ol{S}_{k(m)}\right|+\left|\ol{S}_{k(m)}-\mathbb{E}(\ol{S}_{k(m)})\right| +\left|\mathbb{E}(\overline{S}_{k(m)})-\mathbb{E}(\overline{S}_{n})\right|\\
&\ll m^{\sigma-1}+\left|\mathbb{E}(\overline{S}_{k(m)})-\mathbb{E}(\overline{S}_{n})\right|
\end{align*}
almost surely. Since $m\sim \log^{\frac1\sigma}n$ with $\sigma<1/(B+1)$, there is $m^{\sigma-1}=o(\mathbb{E} (\overline{S}_n))$. Thus, by \eqref{eqsm1} and \eqref{+46}, we have that
\[
\lim_{n\rightarrow\infty}\frac{\overline{S}_n -\mathbb{E} (\overline{S}_n)}{\mathbb{E} (\overline{S}_n)}=0
\]
almost surely, establishing the lemma.
\end{proof}

\section{Binomial \emph{pmf} among arithmetic progressions}
This section is devoted to the proof of Theorems \ref{thmone} and \ref{thmaverageB}. We note that the binomial \emph{pmf} mainly takes values in a short segment, and our argument relies on the study of arithmetic progressions falling into small intervals.
\subsection{Arithmetic progressions falling into small intevals}
\begin{lemma}\label{thmn}
For each $n\in \mathbb{N}$, denote by $I_n$ an interval satisfying $n\in I_n$ and $|I_n|\ll n^\eta\log n$ with $\eta\in(0,1)$. Let $\mA$ be a set of integers satisfying $|\mA(x)|\ll x^{\theta}\log^{-A}x$ for sufficiently large $x$, with $\theta\in(0,1]$. Then for any $\delta\in(0,1)$, we have that
\[
\sum_{x'\le n\le x}\sum_{n^{\delta }\le a\in \mA }\sum_{l\equiv v(\bmod a)}1_{I_n}(l)\ll x^{1+\eta-(1-\theta)\delta}\log^{-A} x
\]
holds uniformly in $x'\in(x^\ve,x)$ and $v\in(-Ax,x'/2)$.
\end{lemma}
\begin{proof}
Denote by
\[
I^x=\mathop{\cup}\limits_{x'\le n\le x}I_n.
\]
Since $x'>2v$, we note that $v\notin I^x$, and thus $l-v\neq 0$ always holds in the sum. 
In virtue of the dyadic partition method, we can split $[x',x]$ into small intervals of forms $[N,2N]$ with $x^\ve\le N\le \frac12x$, and the number of these small intervals is $O(\log x)$.
Thus, it is enough to estimate the sum over each small interval.

For the sum over an interval $[N,2N]$, we swap the order to get
\begin{align}\label{eqsumn}
\sum_{n\sim N}\sum_{n^{\delta }\le a\in \mA }\sum_{l\equiv v(\bmod a)}1_{I_n}(l)
&\ll\sum_{N^{\delta }\le a\in \mA }\sum_{l\equiv v(\bmod a)}\sum_{n\sim N}1_{I_n}(l).
\end{align}
After writing $I^N=\mathop{\cup}\limits_{n\sim N} I_n$, an interval of length $N+O(N^\eta)$, we have
\[
\sum_{n\sim N}1_{I_n}(l)\ll N^\eta\log N \cdot 1_{I^N}(l)
\]
since each $l\in I^N$ can be counted at most $O(N^\eta\log N)$ times in the sum. Applying this into \eqref{eqsumn} gives
\begin{align*}
\sum_{n\sim N}\sum_{n^{\delta }\le a\in \mA }\sum_{l\equiv v(\bmod a)}1_{I_n}(l)\ll N^\eta\log N \sum_{N^{\delta }\le a\in \mA}\sum_{l\equiv v(\bmod a)}1_{I^N}(l).
\end{align*}
For small $N$, direct treatment to the right-hand side of the above formula cannot provide our bound, and we should extend it to a longer interval instead.
Let $I$ be any interval satisfying
\begin{align}\label{eqI}
I^N\subset I\subset I^x,
\end{align}
 and it is obvious that
\begin{align}\label{eqsumalld}
\sum_{n\sim N}\sum_{n^{\delta }\le a\in \mA }\sum_{l\equiv v(\bmod a)}1_{I_n}(l)\ll N^\eta\log N \sum_{N^{\delta }\le a\in \mA }\sum_{l\equiv v(\bmod a)}1_{I}(l).
\end{align}
We would deduce our bound by estimating the sum with an appropriate $I$.
We divide the $a$-sum into two parts, according to whether $a\le |I|$ or not.

For $a\le |I|$, all moduli would not exceed the length of $I$, and by Lemma \ref{lemsumtau}, we have
\begin{align}\label{eqsmalld}
\sum_{\substack{N^{\delta }\le a\le |I|\\ a\in \mA }}\sum_{l\equiv v(\bmod a)}1_{I}(l)&\ll |I| \sum_{\substack{N^{\delta }\le a\le |I|\\ a\in \mA }}\frac{1}{a}\ll |I|N^{-(1-\theta)\delta} \log^{-A} x.
\end{align}
For $a>|I|$, all moduli are greater than the length of the interval, and thus
\begin{align*}
\sum_{l\equiv v(\bmod a)}1_{I}(l)= 0 ~\text{or} ~1.
\end{align*}
If the sum does not vanish, there should be a $l\in I$ of form
\[
l=am+v\ll x.
\]
The condition $a>|I|\ge|I^N|>x^\ve$ implies that
\[
x/m\gg x^\ve,\ \ \ 0<|m| \ll x{|I|}^{-1}
\]
since $v\notin I$.
Applying this, we have
\begin{align}\label{eqbigd}
\sum_{|I|\ll a\in \mA }\sum_{l\equiv v(\bmod a)}1_{I}(l)&\ll \sum_{1\le m\ll x|I|^{-1}}\sum_{a\le x/m,~ a\in \mA}1\\
&\ll (\log x)^{-A} \sum_{1\le m\ll x|I|^{-1}}\left(\frac xm\right)^{\theta}\ll |I|^{-(1-\theta)} x\log^{-A} x.\notag
\end{align}

We deduce our bound from \eqref{eqsmalld} and \eqref{eqbigd} with an $I$ of appropriate length.
Applying \eqref{eqsmalld} and \eqref{eqbigd} with
\[
|I|=\max\left\{x^{\frac1{2-\theta}}N^{\left(\frac{1-\theta}{2-\theta}\right)\delta },N\right\}
\]
into \eqref{eqsumalld}, we have
\begin{align*}
\sum_{n\sim N}\sum_{n^{\delta }\le a\in \mA }\sum_{l\equiv v(\bmod a)}1_{I_n}(l)&\ll N^{\eta-\frac{(1-\theta)^2\delta}{2-\theta} }x^{\frac1{2-\theta}}\log^{-A} x+N^{1+\eta-(1-\theta)\delta}\log^{-A}x\\
&\ll x^{1+\eta-(1-\theta)\delta}\log^{-A}x.
\end{align*}
Summing this over all $N$ establishes the lemma.
\end{proof}

\subsection{Proof of Theorem \ref{thmone}}
In virtue of the identity
\[
\sum_{{l\equiv v (\bmod a)}}  \binom nl \alpha^l (1-\alpha)^{n-l}=\sum_{{l\equiv n-v (\bmod a)}}  \binom nl \alpha^{n-l} (1-\alpha)^{l},
\]
we note that the case $v=n-h$ can be translated to the case $v=h$ at no cost, except with $1-\alpha$ in place of $\alpha$. Thus, we treat the case $v=h$ only and always assume $v\in (-Ax,\alpha x'/2)$ in the rest of this section.
We first deduce a uniformly upper for a sum of binomial \emph{pmf} from Lemma \ref{thmn}.
\begin{lemma}\label{lemubBone}
Suppose by $\mA_x$ a set meeting Condition A with $\theta\in(0,1]$. For any $\alpha, \delta\in (0,1)$, it holds uniformly in $x'\in (x^\ve,x)$ and $v\in (-Ax,\alpha x'/2)$ that
\begin{align}\label{equbB}
\sum_{x'\le n\le x}\mathop{\sum}_{n^\delta\le a\in \mA_x}\sum_{{l\equiv v (\bmod a)}} \binom nl \alpha^l (1-\alpha)^{n-l}\ll x^{1-(1-\theta)\delta}\log^{-A}x.
\end{align}
\end{lemma}
\begin{proof}

Let $l_0=\lfloor\alpha(n+1)\rfloor$ and
\[
I_{l_0}=(l_0-n^{\frac 12}\log n,l_0+n^{\frac 12}\log n).
\]
By Lemma \ref{lemBio}, the rapid decay of the binomial \emph{pmf} implies that
\[
\binom nl \alpha^l (1-\alpha)^{n-l}\ll 2^{-\log^2n}
\]
for $l\notin I_{l_0}$, so the total contribution of these terms is $\ll x^A 2^{-\log^2x'}$, which is negligible since $x'\gg x^\ve$. After applying the uniform upper bound of the binomial \emph{pmf}, it follows that
\begin{align*}
\sum_{x'\le n\le x}\mathop{\sum}_{n^\delta\le a\in \mA_x}\sum_{{l\equiv v (\bmod a)}} \binom nl \alpha^l (1-\alpha)^{n-l}
\ll \sum_{\alpha x'\le l_0\le x}l_0^{-\frac12}\sum_{n^{\delta}\le a\in \mA_x}  \sum_{ l\equiv v \left(\bmod  a\right)}1_{I_{l_0}}(l)
\end{align*}
since $n\ll l_0\ll n$, and every $l_0$ emerges not exceeding $\alpha^{-1}$ times when we sum over $n$.
Applying Lemma \ref{thmn} with $\eta=\frac12$, as well as a summation by parts to the right-hand side, we have
\begin{align*}
\sum_{x'\le n\le x}\mathop{\sum}_{n^\delta\le a\in \mA_x}\sum_{{l\equiv v (\bmod a)}} \binom nl \alpha^l (1-\alpha)^{n-l}\ll x^{1-(1-\theta)\delta}\log^{-A}x,
\end{align*}
establishing the lemma.
\end{proof}

Now we come to the proof of Theorem \ref{thmone}.
We divide  in \eqref{eqaB} the sum over $a$ as
\[
\sum_{a\in \mA_x}=\sum_{a< n^\delta, ~a\in \mA_x}+\sum_{n^\delta\le a\in \mA_x}
\]
with $0<\delta<1$ being a constant to be specified later.
This naturally splits the total sum into two parts, and we write
\begin{align*}
		\sum_{x'\le n\le x}\sum_{a\in \mA_x}\Bigg| \sum_{l\equiv v (\bmod a)} \binom nl \alpha^l (1-\alpha)^{n-l}- \frac 1{a} \Bigg| =  \Sigma_1+\Sigma_2
	\end{align*}
with obvious meanings.

In virtue of Lemma \ref{lemaverage}, we estimate $\Sigma_1$ directly by
\begin{align}\label{eqSigma1}
\Sigma_1\le\sum_{n\le x}  n^{-\frac12}\sum_{a< n^\delta, ~a\in \mA_x}1\ll\sum_{n\le x} n^{-\frac12+\theta\delta}(\log n)^{-A}\ll x^{\frac12+\theta\delta}(\log x)^{-A}.
\end{align}
After removing the absolute value sign in $\Sigma_2$, it follows that
\begin{align*}
\Sigma_2&\le\sum_{x'\le n\le x}\sum_{n^\delta\le a\in \mA_x}\sum_{{l\equiv v (\bmod a)}} \binom nl \alpha^l (1-\alpha)^{n-l}+\sum_{x'\le n\le x}\sum_{n^\delta\le a\in\mA_x}\frac{1}{a},
\end{align*}
where we can bound the first item with Lemma \ref{lemubBone} and the second item with Lemma \ref{lemsumtau}. Thus, we have
\begin{align*}
\Sigma_2\ll x^{1-(1-\theta)\delta}\log^{-A}x.
\end{align*}
This, in combination with \eqref{eqSigma1}, gives
\begin{align*}
\sum_{x'\le n\le x}\sum_{a\in \mA_x}\Bigg| \sum_{l\equiv v (\bmod a)} \binom nl \alpha^l (1-\alpha)^{n-l}- \frac 1{a} \Bigg|  \ll \left(x^{\frac12+\theta\delta}+x^{1-(1-\theta)\delta}\right)\log^{-A}x,
\end{align*}
which establishes the theorem by taking $\delta=\frac12$.

\subsection{Proof of Theorem \ref{thmaverageB}}
In this section, the fact $\theta<1$ and $\ve>0$ always holds, and thus we may ignore the power of $\log x$. Before proving the theorem, we also deduce a uniformly upper bound for the sum of the binomial \emph{pmf} with multiple moduli, but the coprime condition is gone.
\begin{lemma}\label{lemubB}
Suppose by $\mA^1_x,\dots, \mA^u_x$ the sets meeting Condition A with $\theta\in(0,1)$, and denote by $\mD_n$ a set determined by $n$, consisting of $\ll n^\ve$ elements. For any $\alpha, \delta\in (0,1)$ and any given $\nu$,
it holds uniformly in $x'\in (x^\ve,x)$ and $h_i\in (-Ax,\alpha' x'/2)$ that
\begin{align}\label{equbB}
\sum_{x'\le n\le x}\sum_{d\in \mD_n}\mathop{\sum\nolimits}_{\substack {a_1\in\mA_x^1,\dots,a_u\in \mathcal{A}_{x}^u\\ \Pi\bm{a}\gg n^\delta}}\sum_{\substack{l\equiv \bm{v} (\bmod \bm{a})\\ l\equiv \nu (\bmod d)}}  \binom nl \alpha^l (1-\alpha)^{n-l}\ll x^{1-(1-\theta)\delta/u+\ve}
\end{align}
with $\alpha'=\alpha$ for $v_i=h_i$ and $\alpha'=1-\alpha$ for  $v_i=n-h_i$.
\end{lemma}
\begin{proof}
Since all items are nonnegative and $\mD_n\ll n^\ve$, we can remove the $d$-sum first, multiplying by an extra factor $x^\ve$.
By the rapid decay property in Lemma \ref{lemBio}, we may also ignore all the terms with $l\notin I_{l_0}=(l_0-n^{\frac 12}\log n,l_0+n^{\frac 12}\log n)$ as above. Without loss of generality, we assume that $a_1\gg n^{{\delta}/{u}}$ since $\Pi\bm{a}\ge n^\delta$. Note that $l-v_t\neq 0$ for any $l\in I_{l_0}$, which indicates that the number of every $a_t$ is not exceeding $\tau(l-v_t)\ll x^\ve$ for any given $l$.
Thus, we remove the sum over $a_t$, $t\ge2$, with a factor $x^\ve$ after changing the order of the summation, and it follows that
\begin{align}\label{eqsigmaone}
\mathop{\sum\nolimits}_{\substack {a_1\in\mA_x^1,\dots,a_u\in \mathcal{A}_{x}^u\\ \Pi\bm{a}\gg n^\delta}}\sum_{\substack{l\equiv \bm{v} (\bmod \bm{a})}}  \binom nl \alpha^l (1-\alpha)^{n-l}\ll x^\ve\mathop{\sum\nolimits}_{n^{{\delta}/{u}}\ll a_1\in \mathcal{A}_{x}}\sum_{\substack{l\equiv v_1 (\bmod a_1)}}  \binom nl \alpha^l (1-\alpha)^{n-l}.
\end{align}
After applying \eqref{eqsigmaone} into \eqref{equbB}, the lemma follows from Lemma \ref{lemubBone} immediately.
\end{proof}

We now come to the proof of Theorem \ref{thmaverageB}. Let $0<\delta<1$ be a constant to be specified later. We divide the sum into two parts, according to $\Pi\bm{a}\le n^\delta$ and $\Pi\bm{a}> n^\delta$. The treatment for the first part is identical to the proof of Theorem \ref{thmone}. The treatment for the second part proceeds almost identically to the proof of Theorem \ref{thmone}. The differences are applying Lemma \ref{lemubB} instead of Lemma \ref{lemubBone} and multiplying by an extra factor $x^\ve$ due to the $d$-sum and the well-known bound $\tau_u(a)\ll a^\ve$. It follows that
\begin{align}\label{eqtwo}
\sum_{x'\le n\le x}\sum_{d\in \mD_n}\mathop{\sum\nolimits'}_{a_1\in\mA^1_x,\dots,a_u\in \mA^u_x}\Bigg| \sum_{\substack{l\equiv \bm{v} (\bmod \bm{a})\\ l\equiv \nu (\bmod d)}} \binom nl \alpha^l (1-\alpha)^{n-l}- \frac 1{d\Pi\bm{a}} \Bigg|  \ll \left(x^{\frac12+\theta\delta}+ x^{1-(1-\theta)\delta/u}\right) x^\ve.
\end{align}
Thus, taking $\delta=\frac12$ in \eqref{eqtwo} establishes the theorem.

\section{Classical sieved sets by small primes}
This section is devoted to the proof of Theorem \ref{thmsmallp}, and all notations here are as in Section \ref{secint}. The argument relies on Lemma \ref{lemsmm}, the second-moment method in probability, by which the theorem follows immediately after calculating the expectation and the variance with Theorem \ref{thmone}. To apply Theorem \ref{thmone}, we apply here
\begin{align}\label{eqpA}
\mA_x=\bigg\{a\ll x^A: a\mid\prod_{p\le \exp(\log^\beta x)}p\bigg\}.
\end{align}
\begin{lemma}\label{lempA}
For sufficiently large $x$, the set $\mA_x$, defined as in \eqref{eqpA}, meets Condition A with $\theta=1$.
\end{lemma}
\begin{proof}
Let $\omega(a)$ denote the number of distinct prime factors of $a$. By the dyadic partition method, it is sufficient to counter the number of $a\in [y,2y]$ with $y\gg x^\ve$.
Since all prime factors of $a$ are not exceeding than $\exp(\log^\beta x)$, we have
\[
\tau(a)=2^{\omega(a)}\gg 2^{\log y/\log^{\beta}x}.
\]
The prime number theorem gives that
\[
\prod_{p\le \exp(\log^\beta x)}p\ll \exp(\exp(\log^\beta x)).
\]
Hence,
\begin{align*}
|\mA_x\cap[y,2y]|=\sum_{\substack{ a\in \mA^x\\ a\in[y,2y]}}a\tau(a)^{-1}\frac{\tau(a)}{a}&\ll y 2^{-\log y/\log^{\beta}x}\sum_{a\le \exp(\exp(\log^\beta x))}\frac{\tau(a)}{a}\\
&\ll y2^{-\log y/\log^{\beta}x}\exp(2\log^\beta x)\ll y\log^{-A}y
\end{align*}
since $y\gg x^\ve$ and $\beta<1/2$. This establishes the lemma.
\end{proof}
\begin{proposition}\label{proposition11}
With the same condition as in Theorem \ref{thmsmallp}, we have that
	\begin{align*}
		\mathbb{E} \left(\overline{S}_n(\mV^c(n))\right)= c_n+O(\log^{-A}n).
	\end{align*}
\end{proposition}
\begin{proof}
For $l=0,1,\cdots,i$, the probability of $P_i=l$ is $\binom il \alpha^l (1-\alpha)^{i-l}$, and thus
	\begin{align*}
		\mathbb{E}(X_i)=\sum_{0\le l \le i} \binom il \alpha^l (1-\alpha)^{i-l}1_{\mV^c(n)}(l).
	\end{align*}
It is easy to see that
	\begin{align}\label{eqE}
		\mathbb{E} (\overline{S}_n) &=\frac 1n \sum_{n^{1/2}\le i \le n} \mathbb{E} (X_i)+O\left(n^{-\frac12}\right)\\
&=\frac 1n \sum_{n^{1/2}\le i \le n} \sum_{0\le l \le i} \binom il \alpha^l (1-\alpha)^{i-l}1_{\mV^c(n)}(l)+O\left(n^{-\frac12}\right).\notag
	\end{align}
After applying the identity
\begin{align}\label{eqid1}
1_{\mV^c(n)}(l)=\sum_{a\mid l,~a\in\mA_n}\mu(a),
\end{align}
we have
\begin{align*}
		\mathbb{E} (\overline{S}_n) =\frac 1n \sum_{n^{1/2}\le i \le n}\mathop{\sum}_{a\in \mA_n}\mu(a) \sum_{\substack{0\le l \le i \\ l\equiv 0 (\bmod a)}} \binom il \alpha^l (1-\alpha)^{i-l}+O\left(n^{-\frac12}\right),
\end{align*}
and it follows by Theorem \ref{thmone} that
\begin{align*}
		\mathbb{E} (\overline{S}_n)
&=\mathop{\sum}_{a\in \mA_n}\frac{\mu(a)}{a}+O(\log^{-A}n).
\end{align*}
We extend the range of the $a$-sum to infinity with Lemma \ref{lemsumtau} to see
\[
\mathop{\sum}_{a\in \mA_n}\frac{\mu(a)}{a}=c_n+O(\log^{-A}n),
\]
establishing the proposition.
\end{proof}
\begin{proposition}\label{proposition}
With the same condition as in Theorem \ref{thmsmallp}, we have that
	\begin{align*}
		\mathbb{V} \left(\overline{S}_n(\mV^c(n))\right)=O\left(\log^{-A}n\right).
	\end{align*}
\end{proposition}
\begin{proof}
It is well-known that
\begin{align}\label{eqV}
\mathbb{V} (\overline{S}_n)=\mathbb{E} (\overline{S}_n^2)-\mathbb{E} (\overline{S}_n)^2.
\end{align}
By Proposition \ref{proposition11}, there is
\begin{align}\label{eqES21}
\mathbb{E} (\overline{S}_n)^2= c_n^2+O(\log^{-A}n).
\end{align}
For $\mathbb{E} (\overline{S}_n^2)$, we expand the square and take out some terms with the trivial bound, and then it follows that
\begin{align}\label{eqES1}
  \mathbb{E} (\overline{S}_n^2)&=\frac 1{n^2} \sum_{i\le n}\sum_{j \le n} \mathbb{E} (X_i X_j)\\
  &=\frac 2{n^2} \sum_{n^{1/2}\le i\le n-n^\ve}\sum_{n^\ve\le j-i \le n-i} \mathbb{E} (X_i X_j)+O\left(n^{-\frac12}\right).\notag
\end{align}

For $i<j$, the probability of $P_i=s$ with $P_j=s+r$ is
	\begin{align*}
		\binom is \alpha^s (1-\alpha)^{i-s} \binom {j-i}r \alpha^r (1-\alpha)^{j-i-r},
	\end{align*}
and thus
\begin{align}\label{eqEij}
\mathbb{E} (X_i X_j) = &\sum_{{0 \le s \le i}}\binom is \alpha^s (1-\alpha)^{i-s} 1_{\mV^c(x)}(s)\\
 &\times\sum_{{0 \le r \le j-i}} \binom {j-i}r \alpha^r (1-\alpha)^{j-i-r}1_{\mV^c(x)}(s+r).\notag
\end{align}
After inserting \eqref{eqid1}, we apply this into \eqref{eqES1} to see
\begin{align}\label{eqinout}
\mathbb{E} (\overline{S}_n^2) = &\frac 2{n^2} \sum_{n^{1/2}\le i\le n-n^\ve}\sum_{a\in \mA_n}\mu(a)  \sum_{\substack{0 \le s \le i \\ a\mid s}}\binom is \alpha^s (1-\alpha)^{i-s} \\
 &\times\sum_{n^\ve\le j-i \le n-i} \sum_{a'\in\mA_n}\mu(a')  \sum_{\substack{0 \le r \le j-i\\ a'\mid r+s}} \binom {j-i}r \alpha^r (1-\alpha)^{j-i-r}+O\left(n^{-\frac12}\right).\notag
\end{align}
We may replace the two sums of binomial coefficients with $1/a$ and $1/a'$ respectively, and we can bound the difference with Theorem \ref{thmone} after extending both $i$-sum and $j-i$-sum to the range $[n^\ve,n]$, which provides an error term $O(\log^{-A}n)$. Thus, it follows that
\begin{align*}
\mathbb{E} (\overline{S}_n^2)
=& \frac {2}{n^2}\sum_{n^{1/2}\le i\le n-n^\ve}\sum_{a\in \mA_n}\frac{\mu(a)}a  \sum_{n^\ve\le j-i \le n-i} \sum_{a'\in\mA_n}\frac{\mu(a')}{a'}+O\left(\log^{-A}n\right)\notag\\
=&c_n^2+O\left(\log^{-A}n\right)
\end{align*}
by Lemma \ref{lemsumtau},
and this completes the proof.
\end{proof}
\begin{proof}[Proof of Theorem \ref{thmsmallp}]
As $c_n\rightarrow0$, we need to check all the conditions in Lemma \ref{lemsmm} before applying it. It is obvious that
\[
c_n=\prod_{p\le\exp(\log^\beta n)}(1-p^{-1})\gg \log^{-\beta} n
\]
and
\[
\frac{c_n-c_{n-n'}}{c_n}=1-\prod_{\exp(\log^\beta (n-n'))\le p\le\exp(\log^\beta n)}\frac1{(1-p^{-1})}.
\]
While, for $n'\ll n\log^{-B}n$,
\begin{align*}
\sum_{\exp(\log^\beta (n-n'))\le p\le\exp(\log^\beta n)}\log (1-p^{-1})\ll&\sum_{\exp(\log^\beta (n-n'))\le p\le\exp(\log^\beta n)}p^{-1}\\
=&\beta\log\log n-\beta\log \log(n-n')+O(\log^{-\beta}n)\\
\ll&\log^{-\beta}n.
\end{align*}
Thus,
\[
\lim_{n\rightarrow \infty}\frac{c_n-c_{n-n'}}{c_n}=0
\]
holds uniformly for $n'\ll n\log^{-B}n$. Now applying Lemma \ref{lemsmm} with Propositions \ref{proposition11} and \ref{proposition}, we obtain Theorem \ref{thmsmallp} immediately.
\end{proof}
\section{Sieved sets of integers in random walks}
\subsection{The distribution of a sieved set of integers}
From this section, $\mA$ and its cousins $\mA^1$ and $\mA^2$ are sets as in \eqref{eqdefA}, and for large $x$, $\mA_x$ always meets Condition A with $\theta<1$.
Since elements in $\mG$ are coprime with each other, every $a\in \mA$ has a sole factorization
\begin{align}\label{eqdg}
a=g_1g_2\cdots g_s
\end{align}
with $g_i\in \mG$. If $a\in \mA$ owns a factorization \eqref{eqdg}, we define
\[
\mu_{\mG}(a)=(-1)^s.
\]

\begin{lemma}\label{lem1V}
Let $\mV$ be a sieved set defined as in \eqref{eqv1}. For any $n\notin \mH$, we have
\begin{align}\label{eq1V}
1_{\mV}(n)=\sum_{\substack{a_1\mid n-h_1,~ a_1\in \mA\\ (a_1,\Delta_1)\notin\mA-\{1\}}}\cdots\sum_{\substack{a_u\mid n-h_u,~ a_u\in \mA\\ (a_u,\Delta_u)\notin\mA-\{1\},
}}\mu_{\mG}(\Pi\bm{a}),
\end{align}
where $\Delta_1=1$ and $\Delta_t=\prod_{i< t}(h_t-h_i)$ for $t\ge2$.
\end{lemma}
\begin{proof}
It is easy to see the identity
\[
1_{\mV}(n)=\prod_{t\le u}\mathop{\prod}_{\substack{g\mid n-h_t}}\left(1-1_{\mG}(g)\right)=\prod_{t\le u}\mathop{\prod}_{\substack{g\mid n-h_t\\ g\nmid \Delta_t}}\left(1-1_{\mG}(g)\right),
\]
where the condition $g\nmid \Delta_t$ is applied to remove repetitive information from such $g\mid h_i-h_j$ that are common factors of $n-h_i$ and $n-h_j$. Expanding the product, we have

\begin{align*}
1_{\mV}(n)=\prod_{t\le u}\Bigg(\mathop{\sum}_{\substack{a\mid n-h_t,~ a\in \mA\\ (a,\Delta_t)\notin\mA-\{1\}}}\mu_{\mG}(a)\Bigg)=\sum_{\substack{a_1\mid n-h_1,~ a_1\in \mA\\ (a_1,\Delta_1)\notin\mA-\{1\}}}\cdots\sum_{\substack{a_u\mid n-h_u,~ a_u\in \mA\\ (a_u,\Delta_u)\notin\mA-\{1\},
}}\mu_{\mG}(a_1)\cdots\mu_{\mG}(a_u).
\end{align*}
Since elements in $\mG$ are coprime with each other and $(a_i,a_j)\mid (a_j,\Delta_j)$ for any $i<j$, the condition $(a_i,\Delta_i)\notin \mA-\{1\}$ indicates that $(a_i,a_j)=1$, and thus $\Pi\bm{a}\in \mA$ with
\[
\mu_{\mG}(a_1)\cdots\mu_{\mG}(a_u)=\mu_{\mG}(\Pi\bm{a}).
\]
Hence the lemma follows immediately.
\end{proof}

For ease of presentation, we write
\[
\mathop{\sum\nolimits^*}_{a_1,\dots, a_u\in \mA_x}
\]
for the sum over all such $a_t\in \mA_x$, $1\le t\le u$, which are coprime with each other and meet $(a_t,\Delta_t)\notin \mA-\{1\}$. Now we come to the proof of Theorem \ref{pro1v}.
\begin{proof}[Proof of Theorem \ref{pro1v}]
Throughout the proof, the convention $n\notin\mH$ is always assumed.
By Lemma \ref{lem1V}, we have
\begin{align}\label{eqVx}
|\mV(x)|=\mathop{\sum\nolimits^*}_{a_1,\dots, a_u\in \mA_x}\mu_{\mG}(\Pi\bm{a})\sum_{ n\equiv\bm{h} (\bmod \bm{a})}1_x(n)+O(1).
\end{align}
By the Chinese Remainder Theorem, the terms with $\Pi\bm{a}\le y$, where $y\ll x^A$ will be specified later, contribute
\begin{align}
&x\mathop{\sum\nolimits^*}_{\substack{a_1,\dots, a_u\in \mA_x \\ \Pi\bm{a}\le y}} \frac{\mu_{\mG}(\Pi\bm{a})}{\Pi\bm{a}}+O\Bigg(\sum_{a\le y, ~a\in\mA}\tau_u(a)\Bigg)\notag\\
&=x\mathop{\sum\nolimits^*}_{a_1,\dots, a_u\in \mA} \frac{\mu_{\mG}(\Pi\bm{a})}{\Pi\bm{a}}+O\Bigg(x\sum_{y<a\in \mA}\frac{\tau_u(a)}{a}\Bigg)+O\Bigg(\sum_{a\le y,~a\in \mA}\tau_u(a)\Bigg).\notag
\end{align}
The condition $(a_t,\Delta_t)\notin\mA-\{1\}$ indicates that every $g\mid\Pi\bm{a}$ appears at most in one of such $a_t$ with $g\nmid \Delta_t$, and the number of these $a_t$ is equal to $\nu_g(\mH)$.
Thus, we may rewrite the sum in the main term as an Euler product to see that
\[
\mathop{\sum\nolimits^*}_{a_1,\dots, a_u\in \mA}\frac{\mu_{\mG}(\Pi\bm{a})}{\Pi\bm{a}}=\prod_{g\in\mG}\left(1-\frac{\nu_g(\mH)}{g}\right)=c.
\]
Hence it follows that
\begin{align*}
|\mV(x)|&=cx+O\Bigg(\sum_{y<a_1\cdots a_u\in \mA}\sum_{\substack{n\le x\\ n\equiv\bm{h} (\bmod \bm{a})}}1\Bigg)+O\left(xy^{\theta-1+\ve}\right)+O\left(y^{\theta+\ve}\right).
\end{align*}
For the sum in the parentheses, we may assume that $a_1\ge y^{\frac1u}$, without loss of generality. As $n\notin \mH$, we change the order of the summations over $a_t$, $t\ge2$, and $n$ to see
\[
\sum_{y<a_1\cdots a_u\in \mA}\sum_{\substack{n\le x\\ n\equiv\bm{h} (\bmod \bm{a})}}1\ll \sum_{y^{1/u}\le a_1\in \mA}\sum_{n\le x/a_1}\tau(n-h_2)\cdots\tau(n-h_u)\ll xy^{-\frac{1-\theta}{u}+\ve}.
\]
In conclusion, we have
\begin{align*}
|\mV(x)|&=cx+O\left(xy^{-\frac{1-\theta}{u}+\ve}\right)+O\left(xy^{\theta-1+\ve}\right)+O\left(y^{\theta+\ve}\right)\\
&=cx+O\left(x^{\frac{u\theta}{u\theta-\theta+1}+\ve}\right)
\end{align*}
by taking $y=x^{\frac{u}{u\theta-\theta+1}}$, and this establishes the theorem.
\end{proof}

\subsection{Sieved sets of integers in random walks and the proof of Theorem \ref{thm2}}

By Lemma \ref{lemsmm}, the second-moment method in probability, Theorem \ref{thm2} follows directly from the following two propositions, calculations of the expectation and the variance.

\begin{proposition}\label{proposition1}
With the same condition as in Theorem \ref{thm2}, we have, for any $\alpha \in (0,1)$, that
	\begin{align*}
		\mathbb{E} \left(\overline{S}_n(\mV)\right)= c+O\left(n^{-\frac1{2u}(1-\theta)+\ve}\right).
	\end{align*}
\end{proposition}
\begin{proof}
A similar discussion as in the proof of Proposition \ref{proposition11} shows
\begin{align}
		\mathbb{E} (\overline{S}_n) &=\frac 1n \sum_{n^{1/2}\le i \le n} \sum_{0\le l \le i} \binom il \alpha^l (1-\alpha)^{i-l}1_{\mV}(l)+O\left(n^{-\frac12}\right).
\end{align}
Due to the rapid decay of the binomial \emph{pmf}, the contribution of $l\in\mH$ to $\mathbb{E} (\overline{S}_n)$ is negligible.
Then, replacing the indicative function with \eqref{eq1V} gives
\begin{align*}
		\mathbb{E} (\overline{S}_n) &=\frac 1n \sum_{n^{1/2}\le i \le n}\mathop{\sum\nolimits^*}_{a_1,\dots, a_u\in \mA_n}\mu_{\mG}(\Pi\bm{a}) \sum_{\substack{0\le l \le i,~l\notin \mH \\ l\equiv \bm{h} (\bmod \bm{a})}} \binom il \alpha^l (1-\alpha)^{i-l}+O\left(n^{-\frac12}\right).
\end{align*}
Also, due to the rapid decay of the binomial \emph{pmf}, we may remove the condition $l\notin\mH$ in the above sum with a negligible error.
Then, with $(\Pi\bm{a})^{-1}$ in place of the sum of binomial coefficients, it follows that
\begin{align}\label{eqESV}
\mathbb{E} (\overline{S}_n) &=\frac 1n \sum_{n^{1/2}\le i \le n}\mathop{\sum\nolimits^*}_{a_1,\dots, a_u\in \mA_{n}}\frac{\mu_{\mG}(\Pi\bm{a})}{\Pi\bm{a}}+\mathscr{E},
\end{align}
and by Theorem \ref{thmaverageB}, we have
\[
\mathscr{E}=O\left(n^{-\frac1{2u}(1-\theta)+\ve}\right).
\]
For the main term,
\begin{align}\label{eqsuma}
\mathop{\sum\nolimits^*}_{a_1,\dots, a_u\in \mA_{n}}\frac{\mu_{\mG}(\Pi\bm{a})}{\Pi\bm{a}}=\mathop{\sum\nolimits^*}_{a_1,\dots, a_u\in \mA}\frac{\mu_{\mG}(\Pi\bm{a})}{\Pi\bm{a}}+O(n^{-A})=c+O(n^{-A})
\end{align}
by Lemma \ref{lemsumtau}. With this in \eqref{eqESV}, the proposition follows immediately.
\end{proof}
To appeal to the second-moment method, we also need a narrow variance for $\overline{S}_n(\mV)$, which we deduce in the following proposition.
\begin{proposition}\label{proposition2}
With the same condition as in Theorem \ref{thm2}, we have, for any $\alpha \in (0,1)$, that
	\begin{align*}
		\mathbb{V} \left(\overline{S}_n(\mV)\right)=O\left(n^{-\frac1{2u}(1-\theta)+\ve}\right).
	\end{align*}
\end{proposition}
\begin{proof}
An identical argument as in the proof of Proposition \ref{proposition} shows
\begin{align}\label{eqV}
\mathbb{V} (\overline{S}_n)=\mathbb{E} (\overline{S}_n^2)-\mathbb{E} (\overline{S}_n)^2
\end{align}
with
\begin{align}\label{eqES2}
\mathbb{E} (\overline{S}_n)^2= c^2+O\left(n^{-\frac1{2u}(1-\theta)+\ve}\right),
\end{align}
due to Proposition \ref{proposition1}. Then,
\begin{align}\label{eqinout1}
\mathbb{E} (\overline{S}_n^2) = &\frac 2{n^2} \sum_{n^{1/2}\le i\le n-n^\ve}\mathop{\sum\nolimits^*}_{a_1,\dots, a_u\in \mA_n}\mu_{\mG}(\Pi\bm{a})  \sum_{\substack{0 \le s \le i \\ s\equiv \bm{h} (\bmod \bm{a})}}\binom is \alpha^s (1-\alpha)^{i-s} \\
 &\times\sum_{n^\ve\le j-i \le n-i} \mathop{\sum\nolimits^*}_{a_1',\dots, a_u'\in \mA_n}\mu_{\mG}(\Pi\bm{a'})  \sum_{\substack{0 \le r \le j-i\\ r\equiv -s+\bm{h} (\bmod \bm{a'})}} \binom {j-i}r \alpha^r (1-\alpha)^{j-i-r}+O\left(n^{-\frac12}\right),\notag
\end{align}
which we may deduce identically to \eqref{eqinout} except for a discussion on the finite terms with $s\in \mH$ or $r+s\in \mH$ as in the proof of Proposition \ref{proposition1} and an application of \eqref{eq1V} in place of \eqref{eqid1}. By virtue of Theorem \ref{thmaverageB}, we may apply $(\Pi\bm{a})^{-1}$, $(\Pi\bm{a}')^{-1}$ in place of the two sums of binomial coefficients. That is
\begin{align*}
 \mathbb{E} (\overline{S}_n^2) = \frac 2{n^2} \sum_{n^{1/2}\le i\le n-n^\ve}\sum_{n^\ve\le j-i \le n-i}\mathop{\sum\nolimits^*}_{a_1,\dots, a_u\in \mA_n}\frac{\mu_{\mG}(\Pi\bm{a})}{\Pi\bm{a}}\mathop{\sum\nolimits^*}_{a_1',\dots, a_u'\in \mA_n}\frac{\mu_{\mG}(\Pi\bm{a'})}{\Pi\bm{a'}} +O\left(n^{-\frac1{2u}(1-\theta)+\ve}\right).\notag
\end{align*}
Applying \eqref{eqsuma} to the innermost two sums gives
\[
\mathbb{E} (\overline{S}_n^2) =c^2+O\left(n^{-\frac1{2u}(1-\theta)+\ve}\right).
\]
This, in combination of \eqref{eqV} and \eqref{eqES2}, gives the proposition immediately.
\end{proof}

\section{Sieved set of lattice points in random walks}
In this section, we will prove Theorem \ref{thm3}.
Notations in this section coincide with Section \ref{secv2}, and we also apply the notation
\[
\theta=\max\{\theta_1,\theta_2\}
\]
for ease of presentation.
Let $\beta_n$ be a bounded sequence, equally distributed among residue class modulo an integer. That is to say,
for any integer $d$ and large $x$,
\begin{align}\label{eqbetan}
\sum_{\substack{n\le x\\ n\equiv v (\bmod d)}}\beta_n=\frac 1d\sum_{n\le x}\beta_n+O(1).
\end{align}
In this work, we will actually take $\beta_n$ as a polynomial of $n$.

To calculate the expectation and the variance, we require the following lemma.

\begin{lemma}\label{lem2}
Suppose by $\beta_n\ll 1$ a sequence satisfying \eqref{eqbetan}. Let $\bm{h_1}=\{h^1_1,h^1_2,\dots,h^1_u\}$ and $\bm{h_2}=\{h^2_1,h^2_2,\dots,h^2_w\}$ be two given vectors with $h^1_i\in(-Ax,\alpha x'/2)$ and $h^2_j\in(-Ax,(1-\alpha) x'/2)$. Denote by
\begin{align*}
\Xi_2(x)=\sum_{x'\le n \le x}\beta_n \Omega(n)
\end{align*}
with
\begin{align*}
\Omega(n)=\mathop{\sum\nolimits^*}_{a_1,\dots, a_{u}\in \mA^1_x}\mu_{\mG_1}(\Pi\bm{a})  \mathop{\sum\nolimits^*}_{b_1,\dots, b_{w},\in \mA^2_x}\mu_{\mG_2}(\Pi\bm{b}) \sum_{\substack{ l\equiv \bm{h_1} (\bmod \bm{a})\\ l\equiv n-\bm{h_2} (\bmod \bm{b})}} \binom nl \alpha^l (1-\alpha)^{n-l}.
\end{align*}
If $\mA^1_x$ and $\mA^2_x$ meet Condition A with $\theta_1, \theta_2<1$, we have
\[
\Xi_2(x)=c_1c_2\sum_{x'\le n\le x}\beta_n+O\left(x^{1-\frac1{2(u+w)}(1-\theta)+\ve}\right)
\]
holds uniformly in $x'$ with $x^\ve\ll x'< x$.
\end{lemma}
\begin{proof}
Before applying Theorem \ref{thmaverageB}, we should make sure that all moduli in the sum of $\Omega(n)$ are coprime with each other. So, we should separate common factors among coordinates of $\bm{a}$ and $\bm{b}$ first. Denote by
\[
d=(\Pi\bm{a},\Pi\bm{b})=\prod_{i\le u} \prod_{j\le w}(a_i,b_j)
\]
and $\mD_n$ the set of all possible $d$ in the sum of $\Omega(n)$. Since $(a_{i},b_{j})\mid n-h^1_{i}-h^2_{j}$, one notes that there are not too many $d$ for every $n$, actually,
\[
|\mD_n|\ll n^\ve.
\]
In order to keep the rest part still in $\mA^1_x$, we should separate the product $\rho_1(d)$ instead of $d$ from $\Pi\bm{a}$. We can do this by
separating $\rho_1((a_i,d))$ from each $a_i$. We treat $b_j$ identically. For a given $d\in\mD_n$, we write
\[
a'_{i}=\frac{a_{i}}{\rho_1((a_{i},d))},\ \ \ \ \ \ b'_{j}=\frac{b_{j}}{\rho_2((b_{j},d))}.
\]

Since every $g\mid\rho_1(d)$ can appear in one of $\nu_g(\mH_1)$ possible coordinates of $\bm{a}$, we observe that the number of ways to assign $d$ to possible coordinates of $\bm{a}$ is $\prod_{g\mid \rho_1(d)}\nu_g(\mH_1)$.
Thus, we apply the above variable separation and alter the order of the summations to get
\begin{align*}
\Xi_2(x)&=\sum_{x'\le n \le x}\beta_n\sum_{d\in\mD_n}\Bigg(\prod_{g\mid \rho_1(d)}\nu_g(\mH_1)\Bigg)\Bigg(\prod_{g\mid \rho_2(d)}\nu_g(\mH_2)\Bigg)\Omega'(n),
\end{align*}
where
\begin{align*}
\Omega'(n)=\mathop{\mathop{\sum\nolimits^*}_{a'_1,\dots, a'_{u}\in \mA^1_x}\mathop{\sum\nolimits^*}_{b'_1,\dots, b'_{w}\in \mA^2_x}}_{(d,\Pi\bm{a}' \Pi\bm{b}')=(\Pi\bm{a}', \Pi\bm{b}')=1}\mu_{\mG_1}(\Pi\bm{a})\mu_{\mG_2}(\Pi\bm{b}) \sum_{\substack{ l\equiv \bm{h_1} (\bmod \bm{a})\\ l\equiv n-\bm{h_2} (\bmod \bm{b})}} \binom nl \alpha^l (1-\alpha)^{n-l}.
\end{align*}

The sum now is ready to apply Theorem \ref{thmaverageB}. We apply Theorem \ref{thmaverageB} and then extend the sum in the main terms to over all $\mA^1$ and $\mA^2$ with the help of Lemma \ref{lemsumtau}. It follows that
\begin{align}\label{eqdAA'}
 \Xi_2(x)=& \sum_{x'\le n \le x}\beta_n\sum_{d\in\mD_n} d\gammaup_1(d)\gammaup_2(d)\\
 &\times \mathop{\mathop{\sum\nolimits^*}_{a'_1,\dots, a'_{u}\in \mA^1}\mathop{\sum\nolimits^*}_{b'_1,\dots, b'_{w}\in \mA^2}}_{(d,\Pi\bm{a}' \Pi\bm{b}')=(\Pi\bm{a}', \Pi\bm{b}')=1}\frac{\mu_{\mG_1}(\Pi\bm{a})}{\Pi\bm{a}'}\frac{\mu_{\mG_2}(\Pi\bm{b})}{\Pi\bm{b}'}   +O\left(x^{1-\frac1{2(u+w)}(1-\theta)+\ve}\right),\notag
\end{align}
where $\gammaup_i(d)$ is defined in \eqref{eqgamma}.

To alter the order of the summations over $n$ and $d$, we need to reverse the dependence of $d$ on $n$. For a given $d$, both products $\rho_1(d)$ and $\rho_2(d)$ are well-determined, and thus there is a unique factorization
\[
d=\prod_{\substack{g_1\mid \rho_1(d)\\ g_2\mid \rho_2(d)}}(d,g_1,g_2).
\]
If $(d,g_1,g_2)$ is assigned to a couple of possible coordinates, such as the $i$-th coordinate in $\bm{a}$ and the $j$-th coordinate in $\bm{b}$, it holds that
\[
n\equiv h^1_{i}-h^2_{j} \pmod{(d,g_1,g_2)}.
\]
Then, by the Chinese Remainder Theorem, there is a unique $v\in \{0,\dots,d-1\}$ such that
\[
n\equiv v \pmod{d}.
\]
Now we alter the order of the summations over $d$ and $n$ to get
\begin{align*}
\sum_{x'\le n \le x}\beta_n\sum_{d\in\mD_n}d\gammaup_1(d)\gammaup_2(d)&=\sum_d d\gammaup_1(d)\gammaup_2(d)\sum_{\substack{x'\le n \le x\\ n\equiv v(\bmod d)}}\beta_n\\
&=\sum_d d \gammaup_1(d)\gammaup_2(d)\Bigg(\frac{1}d\sum_{x'\le n\le x}\beta_n+O(1)\Bigg)
\end{align*}
since $\beta_n$ meets condition \eqref{eqbetan}.
It is easy to see from Lemma \ref{lemgamma} that the contribution of the error term to $\Xi_2(x)$ is $O(1)$.
 With this in \eqref{eqdAA'}, we combine the sum to recover $a_1\cdots a_{u}=\rho_1(d) a'_1\cdots a'_{u}$ as well as $b_1\cdots b_{w}=\rho_2(d) b'_1\cdots b'_{w}$ to get
\begin{align}
\Xi_2(x)&=\mathop{\sum\nolimits^*}_{a_1,\dots a_{u}\in \mA^1}\frac{\mu_{\mG_1}(\Pi\bm{a})}{\Pi\bm{a}}\mathop{\sum\nolimits^*}_{b_1,\dots b_{w}\in \mA^2}\frac{\mu_{\mG_2}(\Pi\bm{b})}{\Pi\bm{a}} \sum_{x'\le n\le x}\beta_n  +O\left(x^{1-\frac1{2(u+w)}(1-\theta)+\ve}\right)\notag\\
&=c_1c_2\sum_{x'\le n\le x}\beta_n+O\left(x^{1-\frac1{2(u+w)}(1-\theta)+\ve}\right),\notag
\end{align}
establishing the lemma.
\end{proof}
\begin{proposition}\label{prolat1}
With the same condition as in Theorem \ref{thm3}, we have, for any $\alpha\in(0,1)$, that
	\begin{align*}
		\mathbb{E} \left(\overline{S}_n(\mV^2)\right)= c_1c_2+O\left(n^{-\frac1{2(u+w)}(1-\theta)}\right).
	\end{align*}
\end{proposition}
\begin{proof}
A similar discussion as in the proof of Proposition \ref{proposition1} gives
\begin{align}\label{eqESVs}
		\mathbb{E} (\overline{S}_n) &=\frac 1n \sum_{n^{1/2}\le i \le n} \sum_{0\le l \le i} \binom il \alpha^l (1-\alpha)^{i-l}1_{\mV_1}(l)1_{\mV_2}(i-l)+O(n^{-\frac12}).
\end{align}
Applying Lemma \ref{lem1V} twice to the indicative functions and removing the effect of finite terms from $l\in\mH_1$ and $i-l\in\mH_2$ as before, we have
\begin{align*}
		\mathbb{E} (\overline{S}_n) &=\frac 1n \sum_{n^{1/2}\le i \le n}\mathop{\sum\nolimits^*}_{a_1,\dots, a_{u}\in \mA^1_n}\mu_{\mG_1}(\Pi\bm{a})  \mathop{\sum\nolimits^*}_{b_1,\dots, b_{w},\in \mA^2_n}\mu_{\mG_2}(\Pi\bm{b}) \sum_{\substack{l\equiv \bm{h_1} (\bmod \bm{a})\\ l\equiv i-\bm{h_2} (\bmod \bm{b})}} \binom il \alpha^l (1-\alpha)^{i-l}+O(n^{-\frac12}),
\end{align*}
and then the proposition follows immediately by applying Lemma \ref{lem2}.
\end{proof}

\begin{proposition}\label{prolat2}
With the same condition as in Theorem \ref{thm3}, we have, for any $\alpha\in(0,1)$, that
	\begin{align*}
		\mathbb{V} \left(\overline{S}_n(\mV^2)\right)= O\left(n^{-\frac1{2(u+w)}(1-\theta)+\ve}\right).
	\end{align*}
\end{proposition}
\begin{proof}
The steps of the proof is the same as in the proof of Proposition \ref{proposition2}, except applying \eqref{eq1V} four times. It follows that $\mathbb{E} (\overline{S}_n^2)$ is equal to
\begin{align*}
&\frac 2{n^2} \sum_{n^{1/2}\le i\le n-n^\ve}\mathop{\sum\nolimits^*}_{a_1,\dots, a_u\in \mA^1_n}\mu_{\mG_1}(\Pi\bm{a}) \mathop{\sum\nolimits^*}_{b_1,\dots, b_{w},\in \mA^2_n}\mu_{\mG_2}(\Pi\bm{b})
 \sum_{\substack{s\equiv \bm{h_1} (\bmod \bm{a})\\ s\equiv i-\bm{h_2} (\bmod \bm{b})}}\binom is \alpha^s (1-\alpha)^{i-s} \\
 &\times\sum_{n^\ve\le j-i \le n-i} \mathop{\sum\nolimits^*}_{a_1',\dots, a_u'\in \mA^1_n}\mu_{\mG_1}(\Pi\bm{a'}) \mathop{\sum\nolimits^*}_{b_1',\dots, b_w'\in \mA^2_n}\mu_{\mG_2}(\Pi\bm{b'}) \sum_{\substack{r\equiv -s+\bm{h_1} (\bmod \bm{a'})\\ r\equiv j-s-\bm{h_2} (\bmod \bm{b'})}} \binom {j-i}r \alpha^r (1-\alpha)^{j-i-r}
\end{align*}
adding an error term $O\left(n^{-\frac12}\right)$. As $j-s-\bm{h_2}=j-i-(-i+s+\bm{h_2})$ with $-i+s\le 0$, both the inner sum and the outer sum can be calculated with Lemma \ref{lem2}.
The condition \eqref{eqbetan} is required here, and we should calculate the outer sum by applying Lemma \ref{lem2} with $\beta_i=1-n^{-1+\ve}-\frac in$. Thus, we have
\[
\mathbb{E} (\overline{S}_n^2)=c_1^2c_2^2+O\left(n^{-\frac1{2(u+w)}(1-\theta)+\ve}\right).
\]
This, in combination with Proposition \ref{prolat1}, shows
\[
\mathbb{V} (\overline{S}_n)=\mathbb{E} (\overline{S}_n^2)-\mathbb{E} (\overline{S}_n)^2= O\left(n^{-\frac1{2(u+w)}(1-\theta)+\ve}\right),
\]
finishing the proof.
\end{proof}

\section{Sieved sets among visible lattice points in random walks}
Notations in this section coincide with Section \ref{secvl}. For notational convenience, we also apply here two notations
\[
\theta=\max\{\theta_1,\theta_2\},\ \ \ \ \ \vartheta=\max\left\{x^{\frac{u\theta}{u\theta-\theta+1}},x^{\frac{w\theta}{w\theta-\theta+1}}\right\}.
\]

\subsection{The distribution of sieved sets among visible lattice points}
This section is devoted to the proof of Theorem \ref{thm4}.
\begin{proof}[Proof of Theorem \ref{thm4}]
Throughout the proof, the convention $n\notin\mH_1$ and $m\notin\mH_2$ are always assumed, and the contribution of all terms with $n\in\mH_1$ or $m\in\mH_2$ to $\mV^3(x)$ is $O(x)$.
Due to the definition of $\mV^3$, we have
\begin{align*}
	\left|\mV^3(x)\right|&=\mathop{\sum\sum}_{\substack{m,n\le x\\ (m,n)=1}} 1_{\mV_1}(m)1_{\mV_2}(n)+O(x)=\sum_{r\le x}\mu(r)\mathop{\sum\sum}_{\substack{m,n\le x\\ m\equiv n\equiv 0 (\bmod r)}} 1_{\mV_1}(m)1_{\mV_2}(n)+O(x)
\end{align*}
by Lemma \ref{lemmobius}. Applying Lemma \ref{lem1V} to the indicative function, we have
\begin{align}\label{eqNk}
	\left|\mV^3(x)\right|=\sum_{r\le x}\mu(r)\mathop{\sum\nolimits^*}_{a_1,\dots, a_{u}\in \mA^1}\mathop{\sum\nolimits^*}_{b_1,\dots, b_{w}\in \mA^2}\mu_{\mG_1}(\Pi\bm{a})\mu_{\mG_2}(\Pi\bm{b})\mathop{\sum\sum}_{\substack{ m=\bm{h_1} (\bmod \bm{a})\\ n=\bm{h_2} (\bmod \bm{b})\\ m\equiv n\equiv 0 (\bmod r)}} 1_x(m)1_x(n)+O(x),
\end{align}
where the last several sums, after an arrangement, become
\[
\Bigg(\mathop{\sum\nolimits^*}_{a_1,\dots, a_{u}\in \mA^1}\mu_{\mG}(\Pi\bm{a})\sum_{\substack{ m\equiv 0 (\bmod r)\\ m=\bm{h_1} (\bmod \bm{a})}}1_x(m)\Bigg)\Bigg(\mathop{\sum\nolimits^*}_{b_1,\dots, b_{w}\in \mA^2}\mu_{\mG_2}(\Pi\bm{b})\sum_{\substack{ n\equiv 0 (\bmod r)\\ n=\bm{h_2} (\bmod \bm{b})}} 1_x(n)\Bigg)=\lambda_1\lambda_2
\]
with obvious meanings. We just treat $\lambda_1$ in detail, and $\lambda_2$ is identical.

With $r_1=(r,\Pi\bm{a})$, one notes that
\[
\sum_{\substack{ m\equiv 0 (\bmod r)\\ m=\bm{h_1} (\bmod \bm{a})}}1_x(m)=0
\]
unless $r_1\mid \Pi\bm{h_1}$, a convention always holding in the following.
With this condition, every $g\mid\rho_1(r_1)$ only appears in one of such $a_i$ with $(g,r_1)\mid h_i$, and the number of such $a_i$ is $\kappa_1((g,r_1))$. Thus, for a given $r_1$, the number of ways to assign it to possible coordinates of $\bm{a}$ is $\prod_{g\mid \rho_1(r_1)}\kappa_1((g,r_1))$.
With $a_i'=\frac{a_i}{\rho_1((a_i,r_1))}$, it follows that
\[
\lambda_1=\sum_{r_1\mid (r,\Pi\bm{h_1})}\Bigg(\prod_{g\mid \rho_1(r_1)}\left(-\kappa_1((g,r_1))\right)\Bigg)\mathop{\sum\nolimits^*}_{\substack{a'_1,\dots, a'_{u}\in \mA^1 \\ (r,\Pi\bm{a}')=1}}\mu_{\mG_1}(\Pi\bm{a}')\sum_{\substack{ m\equiv 0 (\bmod r)\\ m=\bm{h_1} (\bmod \bm{a})}}1_x(m),
\]
and an identical discussion as for \eqref{eqVx} shows that the last two sums are equal to
\[
\frac{x r_1}{r \rho_1(r_1)}\mathop{\sum\nolimits^*}_{\substack{a'_1,\dots, a'_{u}\in \mA^1 \\ (r,\Pi\bm{a}')=1}}\frac{\mu_{\mG_1}(\Pi\bm{a}')}{\Pi\bm{a}'} +O\left(\left(\frac{x}{r}\right)^{\vartheta+\ve}\right).
\]
Thus, we have
\[
\lambda_1=\frac{x}{r}\sum_{r_1\mid (r,\Pi\bm{h_1})}r_1\omega_1(r_1)\mathop{\sum\nolimits^*}_{\substack{a'_1,\dots, a'_{u}\in \mA^1 \\ (r,\Pi\bm{a}')=1}}\frac{\mu_{\mG_1}(\Pi\bm{a}')}{\Pi\bm{a}'}+O\left(\left(\frac{x}{r}\right)^{\vartheta+\ve}\right)
\]
with $\omega_i(r_i)$ being defined as in \eqref{eqomega}.
Rewrited as an Euler product, the innermost sum evolves into
\[
\prod_{\substack{g\in\mG_1\\(g,r)=1}}\Bigg(1-\frac{\nu_g(\mH_1)}{g}\Bigg)=c_1\prod_{g\mid g_1(r)}\Bigg(1-\frac{\nu_g(\mH_1)}{g}\Bigg)^{-1}.
\]
Thus, we have
\[
\lambda_1=\frac{c_1x}{r}f_1(r)+O\left(\left(\frac{x}{r}\right)^{\vartheta+\ve}\right)
\]
with $f_1$ as in \eqref{eqfi}.

Applying this as well as a similar identity for $\lambda_2$ into \eqref{eqNk} gives
\begin{align*}
\left|\mV^3(x)\right|=&c_1c_2x^2\sum_{r}\frac{\mu(r)}{r^2}f_1(r)f_2(r)+O\Bigg(\sum_{r> x}\frac{x^{2+\ve}}{r^2}\Bigg)+O\Bigg(x^\ve\sum_{r\le x}\Bigg(\frac{x^{1+\vartheta}}{r^{1+\vartheta}}+\frac{x^{2\vartheta}}{r^{2\vartheta}}\Bigg)\Bigg)\\
=&c_1c_2c_3x^2+O\left(x^{1+\vartheta+\ve}\right),
\end{align*}
establishing the theorem.
\end{proof}

\subsection{Sieved sets among visible lattice points in random walks}
In this section, we consider a sieved set among visible lattice points in $\alpha$-random walks. By the second-moment method, we will prove Theorem \ref{thm5} immediately after calculating the expectation and the variance. The proof depends on the following lemma.
\begin{lemma}\label{lem3}
Suppose by $\beta_n\ll 1$ a sequence satisfying \eqref{eqbetan}. Let $\bm{h_1}=\{h^1_1,h^1_2,\dots,h^1_u\}$ and $\bm{h_2}=\{h^2_1,h^2_2,\dots,h^2_w\}$ be two given vectors with $h^1_i\in(-Ax,\alpha x'/2)$ and $h^2_j\in(-Ax,(1-\alpha) x'/2)$. Denoting by
\begin{align*}
\Xi_3(x)=&\sum_{x'\le n \le x}\beta_n\sum_{r\mid n}\mu(r)\Omega(r,n)
\end{align*}
with
\begin{align*}
\Omega(r,n)=\mathop{\sum\nolimits^*}_{a_1,\dots, a_{u}\in \mA^1_x} \mathop{\sum\nolimits^*}_{b_1,\dots, b_{w},\in \mA^2_x}\mu_{\mG_1}(\Pi\bm{a}) \mu_{\mG_2}(\Pi\bm{b}) \sum_{\substack{l\equiv0 (\bmod r) \\ l\equiv \bm{h_1} (\bmod \bm{a})\\ l\equiv n-\bm{h_2} (\bmod \bm{b})}} \binom nl \alpha^l (1-\alpha)^{n-l},
\end{align*}
we have
\[
\Xi_3(x)=c_1c_2c_3\sum_{x'\le n\le x}\beta_n+O\left(x^{1-\frac1{2(u+w)}(1-\theta)+\ve}\right).
\]
\end{lemma}
\begin{proof}
This proof is based on Theorem \ref{thmaverageB}, and we should first separate $\Pi\bm{a}$ and $\Pi\bm{b}$ into several coprime parts to meet the coprime condition in the theorem.
With
\[
r_1=(r,\Pi\bm{a}),\ \ \ \ \ r_2=(r,\Pi\bm{b}),\ \ \ \ \ d=\Bigg(\frac{\Pi\bm{a}}{\rho_1(r_1)},\frac{\Pi\bm{b}}{\rho_2(r_2)}\Bigg),
\]
we separate the variables by writing
\[
a'_{i}=\frac{a_{i}}{(a_{i},\rho_1(dr_1))},\ \ \ \ \  b'_{j}=\frac{b_{j}}{(b_{j},\rho_2(dr_2))}.
\]
Let $\mD_n$ be the set of such $d$. A similar discussion as in the proof of Lemma \ref{lem2} shows that
\[
\left|\mD_n\right|\ll n^\ve.
\]
As discussed before, for a given $r_1$, the number of ways to assign it to possible coordinates of $\bm{a}$ is $\prod_{g\mid \rho_1(r_1)}\kappa_1((g,r_1))$, where $\kappa_i((g,r_i))$ is the number of such $h^i_t\in\mH_i$ that $h^i_t\equiv0 \pmod {(g,r_i)}$.
A similar phenomenon holds for $r_2$. Also, for a given $d\in \mD_n$, the number of ways to assign $d$ to possible coordinates of $\bm{a}$ is $\prod_{g\mid \rho_1(d)}\nu_g(\mH_1)$, and the number of ways to assign $d$ to possible coordinates of $\bm{b}$ is $\prod_{g\mid \rho_2(d)}\nu_g(\mH_2)$.
Thus, we alter the order of the summations to see
\begin{align*}
\Omega(r,n)=&\sum_{r_1\mid r}\Bigg(\prod_{g\mid g_1(r_1)}\left(-\kappa_1((g,r_1))\right)\Bigg)\sum_{r_2\mid r}\Bigg(\prod_{g\mid g_2(r_2)}\left(-\kappa_2((g,r_2))\right)\Bigg)\\
&\times\sum_{d\in\mD_n}\Bigg(\prod_{g\mid \rho_1(d)}\nu_g(\mH_1)\Bigg)\Bigg(\prod_{g\mid \rho_2(d)}\nu_g(\mH_2)\Bigg)\Omega'(r,n),
\end{align*}
where
\begin{align*}
\Omega'(r,n)=\mathop{\mathop{\sum\nolimits^*}_{a'_1,\dots, a'_{u}\in \mA^1_x}\mathop{\sum\nolimits^*}_{b'_1,\dots, b'_{w}\in \mA^2_x}}_{(rd, \Pi\bm{a}'\Pi\bm{b}')=(\Pi\bm{a}', \Pi\bm{b}')=1}\mu_{\mG_1}(\Pi\bm{a}') \mu_{\mG_2}(\Pi\bm{b}') \sum_{\substack{l\equiv0 (\bmod r) \\ l\equiv \bm{h_1} (\bmod \bm{a})\\ l\equiv n-\bm{h_2} (\bmod \bm{b})}} \binom nl \alpha^l (1-\alpha)^{n-l}.
\end{align*}

If regards $rd$ as $d$ in Theorem \ref{thmaverageB}, one notes that the number of possible values of $rd$ is $\ll n^\ve$, and the expression of $\Xi_3(x)$ now is ready to apply the theorem. After applying Theorem \ref{thmaverageB}, we have an analogue of \eqref{eqdAA'} that
\begin{align}\label{eqXi3}
 \Xi_3(x)=&\sum_{x'\le n \le x}\beta_n\sum_{r\mid n}\frac{\mu(r)}{r}\sum_{r_1\mid r}\omega_1(r_1) \sum_{r_2\mid r}\omega_2(r_2)\sum_{d\in\mD_n}d \gammaup_1(d)\gammaup_2(d)\\
&\times\mathop{\mathop{\sum\nolimits^*}_{a'_1,\dots, a'_{u}\in \mA^1}\mathop{\sum\nolimits^*}_{b'_1,\dots, b'_{w}\in \mA^2}}_{(r, \Pi\bm{a}'\Pi\bm{b}')=(\Pi\bm{a}', \Pi\bm{b}')=1}\frac{\mu_{\mG_1}(\Pi\bm{a}')}{\Pi\bm{a}'}\frac{\mu_{\mG_2}(\Pi\bm{b}')}{\Pi\bm{b}'}   +O\left(x^{1-\frac1{2(u+w)}(1-\theta)+\ve}\right),\notag
\end{align}
where $\omega_i(r_i)$ is defined as in \eqref{eqomega}.
A similar discussion as in the proof of Lemma \ref{lem2} shows that we may alter the order of the summations over $d~,r$ and $n$, and then it follows by the Chinese Remainder Theorem that
\begin{align*}
&\sum_{x'\le n \le x}\beta_n\sum_{r\mid n}\frac{\mu(r)}{r}\sum_{r_1\mid r}\omega_1(r_1) \sum_{r_2\mid r}\omega_2(r_2)\sum_{d\in\mD_n}d \gammaup_1(d)\gammaup_2(d)\\
&=\sum_{r\le x}\frac{\mu(r)}{r}\sum_{r_1\mid r}\omega_1(r_1)\sum_{r_2\mid r}\omega_2(r_2)\sum_{d}d \gammaup_1(d)\gammaup_2(d)\Bigg(\frac{1}{dr}\sum_{x'\le n\le x}\beta_n+O(1)\Bigg).
\end{align*}
For $\omega_i(r_i)\le 1$, an direct estimate with Lemma \ref{lemgamma} shows that the contribution of the error term to $\Xi_3(x)$ is $O(1)$.

 With this in \eqref{eqXi3}, we recover $a_1\cdots a_{u}=\rho_1(d) a'_1\cdots a'_{u}$ as well as $b_1\cdots b_{w}=\rho_2(d) b'_1\cdots b'_{w}$ in the sum to remove the coprime condition $(\Pi\bm{a}',\Pi\bm{b}')=1$, and it follows that
\begin{align*}
 \Xi_3(x)=&\sum_{r}\frac{\mu(r)}{r^2}\sum_{r_1\mid r}\omega_1(r_1)\sum_{r_2\mid r}\omega_2(r_2)\\
&\times\mathop{\mathop{\sum\nolimits^*}_{a_1,\dots, a_{u}\in \mA^1}\mathop{\sum\nolimits^*}_{b_1,\dots, b_{w}\in \mA^2}}_{(r, \Pi\bm{a}\Pi\bm{b})=1}\frac{\mu_{\mG_1}(\Pi\bm{a})}{\Pi\bm{a}}\frac{\mu_{\mG_2}(\Pi\bm{b})}{\Pi\bm{b}}\sum_{x'\le n\le x}\beta_n   +O\left(x^{1-\frac1{2(u+w)}(1-\theta)+\ve}\right)\notag\\
=&c_1c_2c_3\sum_{x'\le n\le x}\beta_n +O\left(x^{1-\frac1{2(u+w)}(1-\theta)+\ve}\right),
\end{align*}
establishing the lemma.
\end{proof}

With Lemma \ref{lem3} in place of Lemma \ref{lem2}, identical treatments as in the proof of Proposition \ref{prolat1}, \ref{prolat2} provide the following analogues of the expectation and the variance.
\begin{proposition}
Let $\mV^3$ be a set of visible lattice points given by \eqref{eqv2}, and let $P_i$ be an $\alpha$-random walk as in \eqref{eq2rw}.
We have, for any $\alpha\in(0,1)$, that
\[
\mathbb{E} \left(\overline{S}_n(\mV^3)\right)= c_1c_2c_3+O\left(n^{-\frac1{2(u+w)}(1-\theta)+\ve}\right),
\]
where $c_1$ and $c_2$ are constants given by \eqref{eqc1c2}, and $c_3$ is the constant as in \eqref{eqc3}.
\end{proposition}

\begin{proposition}
Let $\mV^3$ be a set of visible lattice points given by \eqref{eqv2}, and let $P_i$ be an $\alpha$-random walk as in \eqref{eq2rw}.
We have, for any $\alpha\in(0,1)$, that
\[
\mathbb{V} \left(\overline{S}_n(\mV^3)\right)= O\left(n^{-\frac1{2(u+w)}(1-\theta)+\ve}\right).
\]
\end{proposition}

\begin{proof}[Proof of Theorem \ref{thm5}]
With the expectation and the variance above, Theorem \ref{thm5} follows immediately by Lemma \ref{lemsmm}, the second-moment method.
\end{proof}

\end{document}